\documentclass[12pt]{amsart}
\usepackage{amsmath,  amssymb, amscd}
\usepackage{mathtools}
\usepackage[all]{xy}
\usepackage{color}

\usepackage{fullpage,enumerate}
\newcommand\LA{\mathfrak{g}}
\newcommand\PLA{\mathfrak{pg}}

\newcommand\Aut{\operatorname{Aut}}

\theoremstyle{plain}

\newtheorem{thm}{Theorem}[section]
\newtheorem{theorem}{Theorem}[section]
\newtheorem*{thm*}{Theorem}

\newtheorem{corollary}[thm]{Corollary}
\newtheorem*{cor*}{Corollary}

\newtheorem*{conj*}{Conjecture}
\newtheorem*{lemma*}{Lemma}

\newtheorem{lemma}[thm]{Lemma}
\newtheorem*{prop*}{Proposition}

\newtheorem{proposition}[thm]{Proposition}
\theoremstyle{definition}
\newtheorem{rems}[thm]{Remarks}
\newtheorem{remark}[thm]{Remark}
\newtheorem*{defn*}{Definition}
\newtheorem{definition}{Definition}
\newtheorem*{rems*}{Remarks}
\newtheorem*{proof*}{Proof}
\newtheorem{prel*}{Preliminaries}
\newtheorem{examples*}{Examples}
\newcommand\ha{\frac12}

\newcommand{\smA}{\mathcal{S}}
\newcommand{\cK}{{\mathcal K}}

\newcommand{\cR}{{\mathcal R}}
\newcommand{\Index}{{\rm Index}}

\newcommand{\HH}{\operatorname{HH}}

\newcommand{\End}{{\rm{End}}}

\newcommand\bfA{{\bf A}}
\newcommand\bfB{{\bf B}}
\newcommand\bfF{{\bf F}}

\newcommand\bfH{{\bf H}}
\newcommand\bfW{{\bf W}}

\newcommand\bfPhi{{\bf \Phi}}
\newcommand\bfPsi{{\bf \Psi}}
\newcommand\Zeta{\Upsilon}

\newcommand{\PGL}{\operatorname{PG}}
\newcommand{\wX}{{\widetilde X}}
\newcommand{\wF}{{\widetilde {\bfF}}}

\newcommand{\U}{{\sf U}}

\newcommand{\wGamma}{{\widetilde \Gamma}}

\def\wh#1{\widehat{#1}}
\newcommand\Can{\operatorname{Can}}

\newcommand{\bCI}{{\bf \mathcal{C}}^{\infty}}

\newcommand\coF{{}^{\mathcal{C}}\kern-2pt\Lambda}

\newcommand\cFTs{{}^{\Phi}\overline{T}\kern-1pt{}^*}

\newcommand\Out{\operatorname{Out}}

\newcommand\Con{\operatorname{Con}}

\newcommand\WF{\operatorname{WF}}

\newcommand\Tr{\operatorname{Tr}}

\newcommand\rTr{\operatorname{Tr_{R}}}

\newcommand\GL{\operatorname{G}}

\newcommand\cB{\mathcal{B}}

\newcommand\cF{\mathcal{F}}

\newcommand\cG{\mathcal{G}}

\newcommand\cL{\mathcal{L}}

\newcommand{\cP}{{\mathcal P}}

\newcommand\fM{\mathfrak{M}}

\newcommand\CC{\mathbb C}

\newcommand\RR{\mathbb R}

\newcommand\ZZ{\mathbb Z}

\newcommand\bbC{\mathbb C}

\newcommand\bbE{\mathbb E}

\newcommand\bbN{\mathbb N}

\newcommand\bbQ{\mathbb Q}
\newcommand\bbR{\mathbb R}

\newcommand\bbT{\mathbb T}
\newcommand\bbZ{\mathbb Z}

\newcommand\cC{\mathcal C}
\newcommand\cM{\mathcal M}

\newcommand\CIc{{\mathcal{C}}^{\infty}_c}

\newcommand\CI{{\mathcal{C}}^{\infty}}

\newcommand\CmI{{\mathcal{C}}^{-\infty}}

\newcommand\Diag{\operatorname{Diag}}

\newcommand\CF{C_{\Phi}}

\newcommand\cFNs{{}^{\Phi}\overline N\kern-1pt{}^*}

\newcommand\rank{\operatorname{rank}}

\newcommand\Hom{\operatorname{Hom}}
\newcommand\Id{\operatorname{Id}}

\newcommand\UU{\operatorname{U}}
\newcommand\PU{\operatorname{PU}}

\newcommand\ci{${\mathcal{C}}^\infty$}

\newcommand\pa{\partial}

\newcommand\phg{\operatorname{phg}}

\newcommand\supp{\operatorname{supp}}

\newcommand\Mand{\text{ and }}

\newcommand\Mie{\text{ i.e. }}
\newcommand\Mif{\text{ if }}

\newcommand\Mon{\text{ on }}

\newcommand\Mst{\text{ s.t. }}

\newcommand\Mwhere{\text{ where }}
\newcommand\Mwith{\text{ with }}

\begin{document}

\title[Fourier integral operators]
{Geometry of Pseudodifferential algebra bundles\\
and Fourier Integral Operators}

\author{Varghese Mathai}
\address{Department of Pure Mathematics, University of Adelaide,
Adelaide 5005, Australia}
\email{mathai.varghese@adelaide.edu.au}
\author{Richard B. Melrose}
\address{Department of Mathematics,
Massachusetts Institute of Technology,
Cambridge, Mass 02139, U.S.A.}
\email{rbm@math.mit.edu}

\dedicatory{Dedicated to Isadore M. Singer}

\begin{abstract} We study the geometry and topology of (filtered)
  algebra-bundles $\bfPsi^\bbZ$ over a smooth manifold $X$ with typical
  fibre $\Psi^\bbZ(Z; V)$, the algebra of classical pseudodifferential
  operators acting on smooth sections of a vector bundle $V$ over the
  compact manifold $Z$ and of integral order.  First a theorem of
  Duistermaat and Singer is generalized to the assertion that the group of
  projective invertible Fourier integral operators $\PGL(\cF^\bbC(Z; V))$,
  is precisely the automorphism group, of the filtered algebra of
  pseudodifferential operators. We replace some of the arguments in their
  paper by microlocal ones, thereby removing the topological assumption. We
  define a natural class of connections and B-fields on the principal
  bundle to which $\bfPsi^\bbZ$ is associated and obtain a de Rham
  representative of the Dixmier-Douady class, in terms of the outer
  derivation on the Lie algebra and the residue trace of Guillemin and
  Wodzicki; the resulting formula only depends on the formal symbol algebra
  $\bfPsi^\bbZ/\bfPsi^{-\infty}.$ Examples of pseudodifferential algebra
  bundles are given that are not associated to a finite dimensional fibre bundle over $X.$
\end{abstract}

\thanks{{\em Acknowledgments.} The first author was supported by the Australian Research Council
Discovery Project grant DP130103924.  The second author acknowledges the support of
  the National Science Foundation under grant DMS1005944.}
  

\maketitle
\tableofcontents

\section*{Introduction} In this paper, we study the geometry and topology
of (filtered) algebra-bundles $\bfPsi^\bbZ$ over a smooth manifold $X$ with
typical fibre $\Psi^\bbZ(Z; V)$, the algebra of classical
pseudodifferential operators acting on smooth sections of a vector bundle
$V$ over the compact manifold $Z.$ Since we do not assume that there is an
underlying geometric bundle, the transition functions of $\bfPsi^\bbZ$ are
general order-preserving automorphisms of the typical fibre
$\Psi^\bbZ(Z;V).$ By an extension of a theorem of Duistermaat and Singer
\cite{DS} (see Theorem 2.1), every such automorphism is given by
conjugation by an invertible Fourier integral operator of possibly complex
order.  This means in particular that $\Psi^\bbZ(Z; V)$ has many nontrivial
outer automorphisms.  The principal bundle, $\bfF,$ associated to
$\bfPsi^\bbZ$ has fibre the Fr\'echet Lie group of projective invertible
Fourier integral operators, $\PGL(\cF^\bbC(Z;V))=\GL(\cF^\bbC(Z;V))/\bbC^*$
the quotient by the centre, $\bbC^*\Id,$ of the group of invertible Fourier
integral operators of complex order. We show that the structure group can
always be reduced to $\PGL(\cF^0(Z;V)),$ the subgroup of operators of order
$0.$ Thus the structure group arises directly through the central
extension,
\begin{equation}
\bbC^\star\longrightarrow \GL(\cF^0(Z;V))\longrightarrow \PGL(\cF^0(Z;V)).
\label{centralext}\end{equation}

A class of connections on the principal bundle in equation \eqref{centralext} is constructed from the
regularized trace of the Maurer-Cartan 1-form on $\GL(\cF^0(Z;V)).$ The
curvature of such a connection is then computed explicitly via the
trace-defect formula in terms of the residue trace of Guillemin \cite{G85,G93} and
Wodzicki \cite{Wodzicki} giving a differential 2-form on $ \PGL(\cF^0(Z;V)).$ The central
extension \eqref{centralext} is then fixed up to isomorphism by an
additional 1-form, cf. Lemma~\ref{lemma:1-form}.

The obstruction to lifting $\bfF$ to a principal $\GL(\cF^0(Z;V))$-bundle,
and hence to realizing $\bfPsi^{\bbZ}$ as a bundle of operators, is the
Dixmier-Douady invariant (see Definition \ref{def-Cech-DD}). This can be
realized in terms of the bundle gerbe associated to $\bfF$ in the sense of
Murray, as further developed by Murray and Stevenson in \cite{MS}. The
central extension \eqref{centralext} leads to a line bundle over the fibre
product $\bfF^{[2]}$ of the principal bundle with itself and the
Dixmier-Douady invariant is the obstruction to this being obtained as the
difference of the two pull-backs of a line bundle over $\bfF.$

The Chern class of the line bundle over $\bfF^{[2]}$ can be split in terms
of a 2-form, a B-field (or curving), on $\bfF.$ The choice of a connection
on $\bfF$ and a Higgs field that is essentially a section of the adjoint
bundle \footnote{{{A monopole in gauge theory has two types of fields,
      namely a connection and a Higgs field, cf. \cite{JaffeTaubes} which
      is where the terminology originates from. In the context of bundle
      gerbes, it was first introduced in \cite{MS}.}}}, lifting the
exterior derivation on the Lie algebra, enables us to construct an explicit
B-field. The differential of this is a basic differential 3-form
representing the image of the Dixmier-Douady class in de Rham cohomology of
the base, $X.$ 

Recall that Dixmier and Douady \cite{DD} introduced a degree 3 integer
valued cohomological invariant of a $\cK$-algebra bundle (or Azumaya bundle) over a paracompact
space, and moreover proved that every degree 3 integer valued cohomology
class can be realised as the invariant of a $\cK$-algebra bundle over the
paracompact space.  Here $\cK$ denotes the algebra of compact operators on
a separable Hilbert space.  In the case of manifolds, one can consider
algebra bundles of smoothing operators on a compact manifold.  Upon taking
the fibrewise operator norm completion, this gives rise to algebra bundles
of the sort studied by Dixmier-Douady. In particular, we can consider
the Dixmier-Douady
class of the $\cK$-algebra bundle bundle obtained by completing, in the operator norm,
the subbundle $\bfPsi^{-\infty}\subset\bfPsi^\bbZ$ consisting of the
smoothing operators. It is straightforward to see that the Dixmier-Douady class of 
the principal bundle $\bfF$ is equal to the Dixmier-Douady class of the 
associated $\cK$-algebra bundle. Due to the mismatch between the 
topologies of the automorphism group $\Aut(\cK)=\PU$ and $\GL(\cF^0(Z;V))$,
not all degree 3 cohomology classes are realised in our context.  For
instance, it does not seem likely that the context considered by
Freed-Hopkins-Teleman \cite{FHT} can be included in our theory. On the
other hand, the case of torsion Dixmier-Douady invariant was studied in
\cite{MMS1} and the case of non-torsion, decomposable Dixmier-Douady
invariant was studied in \cite{MMS4} and here we study the most general
case that can be realized in terms of pseudodifferential operators as
opposed to abstract operators.

Note that $\bfPsi^\bbZ$, is {\em not} in general associated to a finite
dimensional fibre bundle over $X.$ A continuous section of $\bfPsi^\bbZ$
over $X$ is called a {\em projective family} of pseudodifferential
operators on $Z$ although in view of the conjugation by Fourier integral
operators no meaning can be assigned to the notion of a projective family
of differential operators in this context.

Paycha and Rosenberg (cf. \cite{Paycha-Rosenberg,Paycha}) and others have
considered what amounts to a special case of this general notion of a
bundle of pseudodifferential operators in which the structure group is
required to be the group of invertible pseudodifferential operators. So in
this case the Dixmier-Douady invariant is trivial and there is a bundle of
Fr\'echet spaces on which the pseudodifferential operators act fibrewise.

In outline the content of this paper is as follows. In section
\ref{sect:der} the Lie algebra of derivations of the filtered algebra of
all $\Psi^\bbZ(Z; V)$ is studied, and the Hochschild cohomology of 
$\Psi^\bbZ(Z; V)$ is computed. From this, in section \ref{FormDer} the
structure of the Lie algebra of derivations of the filtered $\star$-algebra
of formal pseudodifferential operators $\Psi^\bbZ(Z; V)/ \Psi^{-\infty}(Z;
V)$ is deduced; in both cases, there are non-trivial outer derivations but
in the formal case the algebra is generally larger. Section~\ref{sect:aut}
is devoted to the study of the Fr\'echet Lie groups of automorphisms of
both the filtered $\star$-algebra of formal pseudodifferential operators
and the filtered algebra of all pseudodifferential operators; again there
are non-trivial outer automorphisms. In Section~\ref{Sect:GFIO} the
topologies on the groups of invertible pseudodifferential and Fourier
integral operator are examined with some additional constructions relegated
to an appendix. The main object of study here, the notion of a filtered
algebra bundle of pseudodifferential over a smooth manifold $X$ is
introduced in Section~\ref{sect:pdobundle}, as well as the existence of
twisted (fibre) cosphere bundles that are not (fibre) cosphere bundles of
fibre bundles. In section \ref{CentralextFIO}, a connection on the central
extension \eqref{centralext} is described with curvature computed in terms
of the residue trace of Guillemin and Wodzicki.  Section~\ref{sect:FIODD}
contains the analysis of the principal bundle of trivializations of a
filtered algebra bundle of pseudodifferential operatorsand the image in deRham
cohomology of the Dixmier-Douady class is computed. Finally several
examples with non-torsion Dixmier-Douady class are given in
Section~\ref{sect:examples}. The appendices contain basic results on the
group of (projective) invertible Fourier integral operators and their
homotopy groups in low dimensions. We use these results to show that there
are infinitely many topologically distinct principal $ \PGL(\cF^0(\Sigma_g))$ bundles 
over $S^n,\, n\ge 2$ that are purely infinite dimensional, and
not arising from any fibre bundles over $S^n,\, n\ge 2$ with typical fibre
$\Sigma_g, \, g\ge 2$.  In Appendix \ref{secappCb}, we study the
construction of pseudodifferential algebra bundles over locally symmetric
spaces and relate them to factors of automorphy and relate their
Dixmier-Douady invariant to these.

We dedicate this paper to our friend and collaborator, Is Singer. We would
like to thank him for his important input in the initial discussions on
this paper and for our earlier work \cite{MMS1, MMS2, MMS3, MMS4}. The
first author thanks M. Murray and D. Stevenson for discussions concerning
their paper \cite{MS}.


\section{Derivations of pseudodifferential operators and Hochschild cohomology}\label{sect:der}


In this section the Lie algebra of derivations on the algebra of classical
pseudodifferential operators acting on sections of a complex vector bundle
$V$ over a compact manifold $Z$, denoted $\Psi^{\bbZ}(Z;V)$, is
characterized.  Equivalently, this computes the first Hochschild cohomology
group of $\Psi^{\bbZ}(Z;V)$. We calculate all the Hochschild cohomology
groups of $\Psi^{\bbZ}(Z;V)$ using results of Brylinski-Getzler \cite{BM}
as well as the H-unitality of $\Psi^{-\infty}(Z;V)$ as explained in
\cite{MN}. Derivations on its `formal' quotient
$\Psi^{\bbZ}(Z;V)/\Psi^{-\infty}(Z;V)$ are considered in the next section.

This analysis is closely parallel to the treatment,
recalled (and refined a little) below, of Duistermaat and Singer of the
group of automorphisms of $\Psi^{\bbZ}(Z;V),$ and in particular is an
infinitesimal version of it. As in that case, it is unnecessary to make any
\emph{a priori} assumption of continuity.  Thus a derivation is simply a
filtered linear map, so for some $k\in\bbZ,$
\begin{equation}
\begin{gathered}
D:\Psi^{p}(Z;V)\longrightarrow \Psi^{p+k}(Z;V),\ \forall\ p\in\bbZ\Mand\\
D(A\circ B)=D(A)\circ B+A\circ D(B).
\end{gathered}
\label{RW.57}\end{equation}
Commutation with an element of the algebra is an inner derivation.

\begin{proposition}\label{Der} If\, $\dim Z\ge2$ and $Z$ is connected, the
quotient of the Lie algebra of derivations of $\Psi^{\bbZ}(Z;V)$ by the
inner derivations is one-dimensional and is generated by $[\log(P),\cdot ],$
where $P$ is a positive pseudodifferential operator  of order one.
Equivalently, the Hochschild cohomology group $\HH^1(\Psi^{\bbZ}(Z;V)) \cong
\CC.$ For $Z=S^1,$ the first Hochschild cohomology group is two dimensional.
\end{proposition}

\begin{proof} 
The first step is to show that a derivation on the subalgebra of smoothing
operators is realized by an operator on $\CI(Z;V).$

\begin{lemma}\label{17.5.2011.1} If
 $D:\Psi^{-\infty}(Z;V)\longrightarrow \Psi^{-\infty}(Z;V)$ is a linear
 map acting as a derivation on smoothing operators,
\begin{equation}
D(A\circ B)=D(A)\circ B+A\circ D(B)\ \forall\ A,\ B\in \Psi^{-\infty}(Z;V),
\label{17.5.2011.2}\end{equation}
then there is a continuous linear operator $L:\CI(Z;V)\longrightarrow
\CI(Z;V)$ such that  
\begin{equation}
D(A)=[L,A]\ \forall\ A\in\Psi^{-\infty}(Z;V)
\label{17.5.2011.3}\end{equation}
and $L$ is determined up to adding a multiple of the identity.
\end{lemma}

\begin{proof} The smoothing operators are naturally identified with the
  smooth sections of the (kernel) bundle $K(V)=V\boxtimes
  (V'\otimes\Omega)$ of the two-point homomorphism bundle over the product
  with a density from the right factor. Thus each derivation may be realized
  as a linear map $D:\CI(Z^2;K(V))\longrightarrow \CI(Z^2;K(V))$ where we
  assume \eqref{17.5.2011.2} but not any continuity. Choose a global
  density, $0<\nu \in\CI(Z;\Omega ),$ and hence an isomorphism $K(V)\simeq
  V\boxtimes V'$ over $Z^2$ so that $D:\CI(Z^2;V\boxtimes
  V')\longrightarrow \CI(Z^2;V\boxtimes V')$ with \eqref{17.5.2011.2}
  holding for the product
\begin{equation}
A\circ B(z,z')=\int_ZA(z,z'')B(z'',z')\nu(z'')\Mon \CI(Z^2;V\boxtimes V').
\label{17.5.2011.4}\end{equation}

In addition choose a smooth fibre inner product $\langle ,\rangle$ on
$V,$ and hence a global inner product on $\CI(Z;V),$  
\begin{equation*}
(u,v)=\int _Z\langle u,v\rangle \nu,\ u,\ v\in\CI(Z;V).
\label{RW.58}\end{equation*}
Each section $w\in\CI(Z;V)$ determines a continuous linear
functional
\begin{equation}
I(w):\CI(X;V)\longrightarrow \bbC,\ I(w)v=\int_Z(v,w)\nu.
\label{17.5.2011.5}\end{equation}

Following the idea of M. Eidelheit \cite{Eidelheit1} we consider smoothing
operators of rank one. Any pair of sections, $u,$ $w\in\CI(Z;V)$ determines
such an operator
\begin{equation}
u\otimes I(w)\in\Psi^{-\infty}(Z;V)\Mwhere (u\otimes I(w))(v)=I(w)(v)\cdot u.
\label{17.5.2011.6}\end{equation}
Note that this is indeed a smoothing operator since its kernel is
$(u\boxtimes w^*)\nu$ where $w^*$ is the `dual' section of $V'$ given by the
pointwise inner product. The resulting map 
\begin{equation}
\CI(Z;V)\times\CI(Z;V)\longrightarrow \Psi^{-\infty}(Z;V)
\label{RW.59}\end{equation}
is linear in the first, but anti-linear in the second,
variable. Conversely, the range of this map consists of all smoothing
operators of rank one. Composition is given by pairing of the two central elements:
\begin{equation}
(u\otimes I(w))\circ(u'\otimes I(w'))=(u',w) u\otimes I(w').
\label{17.5.2011.8}\end{equation}

Now, consider the action of a derivation on these operators. Fix an element
$w\in\CI(Z;V)$ with norm one, i.e.\ $I(w)(w)=1.$ The action of $D$ defines two linear
maps 
\begin{equation}
\begin{gathered}
L,\ R:\CI(Z;V)\longrightarrow \CI(Z;V)\text{ by}\\
Lu=(D(u\otimes I(w)))(w),\ I(Rv)=I(w)\circ D(w\otimes I(v));
\end{gathered}
\label{RW.50}\end{equation}
where in the second case the composite of a smoothing operator and pairing
against a smooth section is necessarily given by pairing against a smooth
section, so in terms of adjoints, 
\begin{equation*}
Rv=\left(D(w\otimes I(v))\right)^*w,\ I(Rv)(\phi)=(D(w\otimes I(v))\phi,w).
\label{RW.60}\end{equation*}

The identity 
\begin{equation*}
u\otimes I(v)=(u\otimes I(w))\circ(w\otimes I(v))
\label{RW.51}\end{equation*}
combined with the derivation property shows that 
\begin{equation}
D(u\otimes I(v))=Lu\otimes I(v)+u\otimes I(Rv)\ \forall\ u,v\in\CI(Z;V).
\label{RW.52}\end{equation}
So $L$ and $R$ determine $D.$

For any four smooth sections, expanding out the composition formula
\eqref{17.5.2011.8} applied to $(u\otimes I(v))\circ (u'\otimes I(v'))$ and using 
\eqref{RW.52} gives, 
\begin{equation}
\begin{gathered}
(u',v)\left(Lu\otimes I(v')+u\otimes I(R v')\right)\\
=(u',v)Lu\otimes I(v')+(u',Rv)u\otimes I(v')\\
+(Lu',v)u\otimes I(v')+(u',v)u\otimes I(Rv').
\end{gathered}
\label{RW.53}\end{equation}
Thus the middle two terms on the right must cancel. This gives the adjoint
identity $(u',Rv)=-(Lu',v)$ for all $u'$ and $v,$ i.e.\ $R=-L^*$ and hence
\eqref{RW.52} becomes 
\begin{equation}
\begin{gathered}
D(u\otimes v)=L\circ(u\otimes I(v))-(u\otimes I(v))\circ L,\ \Mie\\
D(A)=[L,A],\ A\text{ of rank one}.
\end{gathered}
\label{RW.54}\end{equation}
By linearity this extends to all operators of finite rank and more
generally, if $E$ is any smoothing operator and $A,$ $B$ are of finite rank
then so is $AEB$ and the derivation identity shows that 
\begin{equation}
D(AEB)=[L,A]EB+AD(E)B+AE[L,B]
=[L,A]EB+A[L,E]B+AE[L,B]
\label{RW.55}\end{equation}
so $D(E)=[L,E]$ for all smoothing operators since any operator is
determined by the collection of the $AEB$ with $A$ and $B$ of rank one.

To see that $L:\CI(Z;V)\longrightarrow \CI(Z;V)$ is continuous observe that
the discussion above shows that, without assuming continuity, $L$ has a
well-defined adjoint, namely $-R.$ Thus $L:\CI(Z;V)\longrightarrow
\CI(Z;V)$ is closed, since if $u_n\to u$ and $Lu_n\to w$ then 
\begin{equation}
(w,\phi)=\lim (Lu_n,\phi)=-\lim(u_n,R\phi)=-(u,R\phi)=(Lu,\phi).
\label{RW.61}\end{equation}
As a closed linear operator on a Fr\'echet space, $L$ is necessarily continuous.

The uniqueness of $L$ up to the addition of a scalar multiple of the
identity follows from the fact that these are the only 
operators which commute with all smoothing operators.
\end{proof}

Now consider a filtration-preserving derivation on $\Psi^{\bbZ}(Z;V)$ as in
\eqref{RW.57}. It follows that it induces a derivation on
$\Psi^{-\infty}(Z;V),$ being the intersection of these spaces, and
Lemma~\ref{17.5.2011.1} generates an operator $L:\CI(Z;V)\longrightarrow
\CI(Z;V).$ Moreover, the identity \eqref{RW.55} again shows that
\begin{equation}
D(A)=[L,A]\ \forall\ A\in\Psi^{\bbZ}(Z;V).
\label{17.5.2011.17}\end{equation}

As an operator, $L$ determines and is determined by its Schwartz' kernel,
which we also denote $L\in\CmI(Z^2;V\boxtimes V)).$ If
$A\in\Psi^{\bbZ}(Z;V)$ it acts on $\CI(Z;V)$ and its formal adjoint is an
element $A^t\in\Psi^{\bbZ}(Z;V).$ Then the identity \eqref{17.5.2011.17}
can be written in terms of the kernel
\begin{equation}
(A\otimes\Id-\Id\otimes A^t)L=B,\ B\in\Psi^{\bbZ}(Z;V)
\label{17.5.2011.19}\end{equation}
also representing the Schwartz kernel of the operator $D(A)$. In general neither
term on the left here is a pseudodifferential operator on $Z^2.$ However,
if $A$ is a differential operator, say of order $1,$ then so is $A^t$ and
then \eqref{17.5.2011.19} represents a differential equation on $Z^2.$

It follows from \eqref{17.5.2011.19} that $L$ itself has wavefront
contained in the conormal bundle to the diagonal, since at any other
point in $T^*Z^2\setminus\{0\}$ it is possible to choose $A$ so that
$(A\otimes\Id-\Id\otimes A^t)$ is elliptic. Indeed, if
$(z,\zeta,z',\zeta')\in T^*Z^2=(T^*Z)^2$ is a non-zero vector where
$z\not=z'$ then either $A$ can be chosen to be elliptic at $z$ and vanish
near $z'$ or conversely and then $A-A^t$ is elliptic at this point. If
$z=z'$ but $\zeta\not=-\zeta'$ with one of these non-zero then $A$ can be
chosen to be elliptic at one point and characteristic at the other, making
$A-A^t$ microlocally elliptic.

In particular the kernel of $L$ is smooth away from the diagonal. Cutting
it off appropriately, $L$ can be decomposed into the sum of a smoothing
operator and an operator with kernel supported in a preassigned
neighbourhood of the diagonal. The smoothing term gives an inner derivation so
it is enough to suppose that $L$ has
support near the diagonal and then it is readily analysed in local
coordinates. It is enough to consider its action as a map from sections
supported in a coordinate patch over which $V$ is trivial, into sections
on the same coordinate patch, for a finite covering of $Z$ by coordinate
charts. In such coordinates $z,$ $L$ can be written in Weyl form
\begin{equation}
L(z,z')=(2\pi)^{-n}\int g\left(\frac{z+z'}2,\zeta\right)e^{iz\cdot(z-z')}d\zeta |dz'|
\label{18.5.2011.4}\end{equation}
where $g\in\CI(\Omega \times\bbR^n)\otimes M(N;\bbC)$, $\Omega$ 
is the domain of a coordinate chart and $g$ is polynomially bounded in $\zeta$. The
smoothness in the base variables follows from the restriction on the
wavefront set obtained above and the smoothness in the fibre variables from
the compactness of the support in the normal direction to the diagonal,
i.e.\ in $z-z'.$ Note that the function $g$ is well-defined locally.

Now the conjugation condition \eqref{17.5.2011.19} implies that for $\Omega'$ 
a subset of $\Omega$,
\begin{equation}
\begin{gathered}
[L,z_j] \in\Psi^{k}(\Omega';\bbC^N),\\
[L,D_{z_j}]\in \Psi^{k+1}(\Omega';\bbC^N)
\end{gathered}
\label{18.5.2011.5}\end{equation}
since in both cases the differential operators $z_j$ and $D_{z_j}$ can be
cut off very close to the boundary of the coordinate patch and then
\eqref{18.5.2011.5} holds in some slightly smaller domain $\Omega'.$ Since
the test operators are local, the kernels on the right are supported in the
same neighbourhood of the diagonal as the kernel of $L.$ Thus the
pseudodifferential operators can also be written locally uniquely in Weyl
form \eqref{18.5.2011.4} and it follows from this uniqueness that 
\begin{equation} 
\begin{gathered}
D_{\zeta_j}g\in S^{k}_{\phg}(\Omega '\times\bbR^n)\otimes M(N,\bbC),\\
D_{z_j}g\in S^{k+1}_{\phg}(\Omega '\times\bbR^n)\otimes M(N,\bbC)
\end{gathered}
\label{18.5.2011.6}\end{equation}
where the spaces on the right consist of the (matrix-valued) classical
symbols of some integral order. From the first of these it follows that 
\begin{equation}
\zeta\cdot\pa_{\zeta}g=h\in S^{k+1}_{\phg}(\Omega '\times\bbR^n)\otimes M(N,\bbC).
\label{18.5.2011.8}\end{equation}

This differential equation is easily solved near infinity in $\zeta.$
Namely, each of the terms which are homogeneous of non-zero degree on the
right can be solved away by a multiple on the left. Taking an asymptotic sum
of these terms gives an element $g'\in S^{k+1}_{\phg}(\Omega
'\times\bbR^n)\otimes M(N,\bbC)$ such that  
\begin{equation}
\zeta\cdot\pa_{\zeta}g'=h-h_0-h'',\ h''\in S^{-\infty}(\Omega '\times\bbR^n)\otimes M(N,\bbC)
\label{18.5.2011.9}\end{equation}
and where $h_0$ is homogeneous of degree $0$ in $|\zeta|>1.$ The rapidly
decaying term can be integrated away radially to give a rapidly decaying
solution of $\zeta\cdot\pa_{\zeta}g''=h''-r$ where $r$ has support in
$|\zeta|<1.$ It follows that  
\begin{equation}
g=g'+g''+\log|\zeta|\cdot h_0(z,\zeta)+g_0(z,\zeta)\text{ in }|\zeta|\ge 1
\label{18.5.2011.10}\end{equation}
where all terms are smooth and $g_0$ is homogeneous of degree $0$ in
$\zeta.$ Substituting this back into \eqref{18.5.2011.6} -- and noting that
all other terms are classical -- it follows that 
\begin{equation}
D_{\zeta_j}h_0(z,\zeta)=0,\ D_{z_j}h_0(z,\zeta)=0.
\label{18.5.2011.11}\end{equation}
Thus in fact $h_0$ is constant provided the cosphere bundle is connected,
i.e.\ $Z$ is connected and not the circle. Since the commutator with all constant
matrices must also be classical it follows that $h_0$ must be a constant
multiple of the identity matrix
\begin{equation}
g(z,\zeta)=c\Id\log|\zeta|+\tilde g,\ \tilde g\in S^{k+1}_{\phg}(\Omega
'\times\bbR^n)\otimes M(N,\bbC)\text{ in }|\zeta|>1.
\label{18.5.2011.12}\end{equation}

Now, consider some positive elliptic second order differential operator
with scalar principal symbol acting on sections of $V$ and take its complex
powers $P^z$, see \cite{Seeley1}. Then $\log P=dP^z/dz$ at $z=0$ is a
globally defined pseudodifferential operator which, whilst non-classical,
acts as a derivation on the classical operators since conjugation by $P^z$
maps $\Psi^{m}(\Omega;\bbC^N)$ to itself. Moreover the Weyl symbol of $\log
P$ is precisely of the form of a classical symbol (of order $0)$ plus
$\log|\zeta|$ in any local coordinates. It follows that everywhere locally
\begin{equation}
L-c\log P\in \Psi^{k+1}(Z;V)
\label{18.5.2011.13}\end{equation}
and hence this is globally true. This completes the proof of Proposition~\ref{Der}.
\end{proof}

We calculate the Hochschild cohomology groups of $\Psi^{\bbZ}(Z;V)$
using results of Brylinski-Getzler \cite{BG} as well as the H-unitality of
smoothing operators $\Psi^{-\infty}(Z;V)$ as explained in \cite{MN}.  More
precisely, the short exact sequence
\begin{equation}
\xymatrix{
0\ar[r]&\Psi^{-\infty}(Z;V)\ar[r]& \Psi^{\bbZ}(Z;V)\ar[r]& \Psi^{\bbZ}(Z;V)/\Psi^{-\infty}(Z;V)\ar[r]&0,
}
\end{equation}
gives rise to a long exact sequence of Hochschild cohomology groups
\begin{equation}
\xymatrix{
0\ar[r]& \HH^0(\Psi^{\bbZ}(Z;V)/\Psi^{-\infty}(Z;V))  \ar[r]&\HH^0(\Psi^{\bbZ}(Z;V)) \ar[r]&
\HH^0(\Psi^{-\infty}(Z;V)) \ar[r]^{\qquad\partial} & \ldots,
}
\end{equation}
where $\partial$ denotes the boundary operator.  Since
$\Psi^{-\infty}(Z;V)$ is H-unital (see \cite{Wodzicki2}), it is Hochschild
cohomologous to a matrix algebra, so by Morita invariance of Hochschild
cohomology, $\HH^k(\Psi^{-\infty}(Z;V)) = \HH^k(\CC) = \{0\},$ for all $k>0.$
Therefore we conclude that $\HH^k(\Psi^{\bbZ}(Z;V)) \cong
\HH^k(\Psi^{\bbZ}(Z;V)/\Psi^{-\infty}(Z;V)),$ for all $k\ne1.$ The right
hand side has been computed in \cite{BG} so we have proved:

\begin{theorem}
If\, $\dim Z\ge2$ and $Z$ is connected, then the Hochschild cohomology groups of $\Psi^{\bbZ}(Z;V)$ are,
\begin{equation}
\HH^k(\Psi^{\bbZ}(Z;V))
\cong \begin{cases} \bfH^k(S^*Z \times S^1, \CC) &\text{if }k\ne1,\\
 \CC& \text{if }k=1.\end{cases}
\end{equation}
\end{theorem}


\section{Derivations of formal pseudodifferential operators}\label{FormDer}


In the same setting as the preceding section, with $Z$ a compact manifold
and $\Psi^m(Z;V)$ denoting the space of all classical pseudodifferential
operators of order $m$ on $Z$ acting on sections of a complex vector bundle
$V$ over $Z,$ the quotient $\cB^{\bbZ} = \cB^{\bbZ}(T^*Z;\hom
V)=\Psi^{\bbZ}/\Psi^{-\infty}(Z;V)$ is the space of formal pseudodifferential
operators, also called the full symbol algebra. It may be identified by a
(non-canonical) choice of quantization with the space of `Laurent' series
of infinite sums of homogeneous sections, of integral degree, of $\hom(V)$
over $T^*Z\setminus0$ with homogeneity bounded above but not below. It is
then a star algebra in the sense that the product is the local bundle
composition at top level of homogeneity with the second term, when the
bundle is locally trivialized and Weyl quantization is chosen, given by the
Poisson bracket extended to matrices. The algebra acts on itself as a Lie
algebra of derivations with only multiples of the identity acting trivially.

\begin{proposition}\label{ForDer} If $Z$ is compact and connected, the
  space $\Out(\cB^\bbZ)$ of filtered outer derivations on $\cB^{\bbZ},$ the
  quotient of (algebraic) derivations by inner derivations, fits in the short
  exact sequence,
\begin{equation}
\xymatrix{
0\ar[r]&\Out(\Psi^{\bbZ})\ar[r]& \Out(\cB^{\bbZ})\ar[r]& \bfH^1(S^*Z;\bbC)\ar[r]&0.
}
\label{20.5.2011.3}\end{equation}
Equivalently, one has a (split) short exact sequence of Hochschild cohomology groups,
\begin{equation}
\xymatrix{
0\ar[r]&HH^1(\Psi^{\bbZ})\ar[r]& HH^1(\cB^{\bbZ})\ar[r]& \bfH^1(S^*Z;\bbC)\ar[r]&0.
}
\label{20.5.2011.3a}\end{equation}
\end{proposition}

\noindent This sequence serves to explain the simplifying assumption made by
Duistermaat and Singer \cite{DS} that $\bfH^1(S^*Z;\CC)=0$, as will be clear in  the proof.

\begin{proof} The filtered derivations on $\Psi^{\bbZ}(Z;V)$ certainly
  induce such derivations on the quotient, with inner derivations mapped to
  inner derivations, so the first map in \eqref{20.5.2011.3} is
  well-defined. It is also clearly injective from the 1-dimensional space
  generated by $[\log Q,\cdot]$ or from the two-dimensional case for the
  circle. We proceed to characterize all the filtered derivations on the
  formal symbol algebra.

Choosing a metric on $Z,$ the real powers of the metric function on
$T^*Z\setminus0$ are homogeneous of any given degree. This allows the
leading part of a derivation $D:\cB^{\bbZ}\longrightarrow \cB^{\bbZ+k}$ to
be normalized to map $D_k:\CI(S^*Z;\hom V)\longrightarrow \CI(S^*Z;\hom V)$
which is again a derivation for the matrix product. Directly from the
definition such a map is local with the value at any given point only
depending on the 1-jet, i.e.\ is given by a linear differential operator of
first order. In fact, all terms in the star product are local with
dependence on increasing order of jets, so evaluated at each point the
derivation is given by a linear differential operator on each term in the
full symbol. To examine this operator it suffices to work in local
coordinates and in terms of a local trivialization of $V.$ Thus $D_k,$ the
leading term of the derivation, is given by a matrix of first order
differential operators.

The constant term in $D_k$ is determined by evaluating the derivation on
constant sections of $\hom V=M(N,\bbC)$ near the point. When either factor
is constant the star product is just the matrix product. Necessarily the
constant term is then a derivation on $\hom V_p=\bbC^N$ at each point and
therefore is given by commutation with a matrix which is uniquely
determined up to addition of a multiple of the identity. The choice of a
matrix of trace zero is therefore unique. Since the derivation maps smooth
sections to smooth sections this defines a smooth section of $\CI(S^*Z;\hom
V).$ Composing with the appropriate power of the metric used to normalize
the leading part of the derivation above, this gives the leading term of a
symbol, which therefore gives an inner derivation with the same constant
term. Subtracting this from the original derivation gives a derivation with
leading term $D_k$ which annihilates the constant matrices of order $0.$  It follows
that this leading derivation is given by a section of $T(S^*Z)\otimes\hom V$
with the first factor acting as a vector field in local coordinates. In
fact it must then distribute over multiplication, at each point, by
constant matrices
\begin{equation*}
D_k(uu_0)(p)=(D_k(u))u_0(p),\ \nabla u_0(p)=0
\end{equation*}
at that point. So in fact $D_k$ is reduced to a differential operator with
scalar principal symbol which vanishes on the constants, i.e.\ a vector field
acting as a multiple of the identity. 

The action of $D$ on sections of any integral homogeneity $m$ can be
deduced by composition with the appropriate power of the metric function $g.$ Thus
on symbols of order $m$ the leading part of the derivation of order $m+k$ is
\begin{equation*}
D(u)_k\equiv(D(u)g^{-m-k})_0
\equiv D(ug^{-m})\cdot g^{-k}-mua^{-m}(g^{-k-1}Dg)\equiv
D_k(ug^{-m})-mbu
\end{equation*}
in terms of equality of the leading parts of symbols of order $0$ and with
$b$ an element of $\CI(S^*M).$ Thus in general the leading part of $D$ as a
map from homogenous sections of degree $m$ to homogeneous sections of degree
$m+k$ is given by a scalar vector field which is homogeneous of degree $k.$

Next we analyse the second term in homogeneity of the derivation. Working
now up to error of relative homogeneity $2$ the derivation induces a second map
$$
D(u)=D_ku+S_{k-1}(u)
$$
where $S_{k-1}$ increases homogeneity (of each term) by $k-1.$ In terms of
Weyl quantization the star product, up to second order, is

\begin{equation}
u*v=uv-\frac i2\{u,v\}
\label{RW.11}\end{equation}
where $\{u,v\}$ is the Poisson bracket acting in the components of matrix
multiplication. So the derivation identity on the product of two scalar
elements and a constant matrix, $uvE$ can be written
\begin{equation*}
\begin{gathered}
D_k(uv)E =D_k(u)vE+uD_k(v)E\\
\begin{aligned}
S_{k-1}(uvE)-S_{k-1}(u)vE-&uS_{k-1}(vE)=\\
\frac i2\big(D_k\{u,v\}E&-\{D_ku,v\}E-\{u,D_kv\}E\big).
\end{aligned}
\end{gathered}
\end{equation*}
The left side certainly involves no more than three derivatives in total,
but actually is only of second order. Taking $u(p)=v(p)=0$ allows the
principal symbol of this second order differential operator to be computed
-- the left side is therefore symmetric in the first derivatives of $u$ and
$v$ whereas the right side is clearly antisymmetric. So in fact both sides
are of order $1$ and since the right side annihilates constants in either
$u$ or $v$ it must vanish identically (and $S_{k-1}$ must itself be a
derivation). The resulting identity is precisely the condition that $D_k$
be a symplectic vector field, i.e.\ one which distributes over the Poisson
bracket, and hence is of the form
$$
\omega(\cdot, D_k)=\alpha ,\ d\alpha =0
\label{RW.8}
$$
where $\alpha$ is a closed form on $T^*Z\setminus0$ which is homogeneous of
degree $k.$

For any degree other than $0$ the closed homogeneous forms on
$T^*Z\setminus0$ are exact. In degree $0$ such a form is the sum of the
pull-back from $S^*Z$ of a closed form, plus a multiple of the closed form
$g^{-1}dg$ given by the metric function. The latter is exact in the sense
that it is $d\log g$ which corresponds precisely to the derivation given by
$\log Q.$ The exact forms arise from inner derivations.

That elements of $\bfH^1(S^*Z,\bbC)$ do correspond to derivations on
$\cB^{\bbZ}(T^*Z;\hom V)$ can be seen for instance by passing to the
universal cover $\tilde Z$ of $Z.$ Since the formal symbol algebra
corresponds to localization at the diagonal it can be identified with the
$\pi_1$-invariant part of the quotient of the properly supported
pseudodifferential operators on $\tilde Z$ by the properly supported
smoothing operators. On $T^*\tilde Z\setminus 0$ every closed form is exact
and the elements of $\bfH^1(S^*Z;\bbC)$ correspond to smooth functions on
$S^*\tilde Z$ which are $\pi_1$-invariant up to shifts by constants. These
can be realized as multiplication operators, hence as properly supported
pseudodifferential operators, on $\tilde Z,$ commutation with which induces
$\pi_1$-equivariant derivations on the formal symbol algebra, and hence
fully $\pi_1$-invariant derivations on the invariant subalgebra,
i.e.\ derivations on $\cB^{\bbZ}(T^*Z;\hom V).$

This completes the proof of \eqref{20.5.2011.3} and hence the Proposition.
\end{proof}


\section{Automorphisms of pseudodifferential operators}\label{sect:aut}


Next we consider the group of order-preserving automorphisms of the
classical pseudodifferential algebra; this group was characterized by
Duistermaat and Singer. We recall and somewhat extend the main theorem from
\cite{DS}. If $\chi:S^*Z\longrightarrow S^*Z'$ is a contact
transformation between two compact manifolds, which is to say a canonical
diffeomorphism between their cosphere bundles, let $\cF^s(\chi)$ denote the
linear space of Fourier integral operators associated to $\chi$ of complex
order $s;$ thus each $F\in\cF^s(\chi)$ is a linear operator
$F:\CI(Z)\longrightarrow \CI(Z')$ which has Schwartz kernel which is a
Lagrangian distribution with respect to the twisted graph of $\chi$
(\cite{HoFIO}). For the convenience of the reader a very brief discussion of
Fourier Integral operators can be found in the appendix.

\begin{theorem}\label{thm:DS} For any two compact manifolds $M_1$ and $M_2,$
  and complex vector bundles $V_1$ and $V_2$ over them, every linear
  order-preserving algebra isomorphism $\Psi^{\bbZ}(M_1;V_1)\longrightarrow
  \Psi^{\bbZ}(M_2;V_2)$ is of the form
\begin{equation}
\Psi^{\bbZ}(M_1;V_1)\ni A\longrightarrow FAF^{-1}\in\Psi^{\bbZ}(M_2;V_2)
\label{Duistermaat-Singer.2}\end{equation}
where $F\in\cF^s(\chi;V_1,V_2)$ is a classical Fourier integral operator of
complex order $s$ associated to a canonical diffeomorphism
$\chi:T^*M_1\setminus0\longrightarrow T^*M_2\setminus0$ and having inverse
$F^{-1}\in\cF^{-s}(\chi^{-1};V_2,V_1);$ $F$ is determined
by \eqref{Duistermaat-Singer.2} up to a non-vanishing multiple of the
identity.
\end{theorem}

\noindent This is the result of \cite{DS} except the restriction that
$\bfH^1(S^*Z, \CC)=\{0\}$ is removed; for simplicity of presentation the case of
non-compact manifolds is not considered here, but on the other hand the
action on sections of vector bundles is included. The proof is also
essentially that of \cite{DS} with some rearrangement; it is closely
parallel to the discussion of derivations above. The only significant
differences from \cite{DS} are the use of a microlocal regularity argument
in place of some of the more constructive methods in the original and an
argument using spectral theory to eliminate the `anticanonical'
possibility. Note that in general there are \emph{no} invertible Fourier
integral operators between two manifolds. For such operators to exist, the
manifolds must certainly have the same dimension, the cosphere bundles must
be contact-diffeomorphic and the vector bundles must also have the same
rank. There is also an index obstruction, \cite{EM}, \cite{LN}.

\begin{remark}\label{PostArx.4} No continuity needs to be assumed of the
  automorphism. However, if one considers a finite dimensional family of
  automorphisms which are continuous or smooth in the sense that the image
  of any given pseudodifferential operator is smooth in the parameters in
  terms of the Fr\'echet topology on pseudodifferential operators then the
  proof below shows that the resulting family of projective Fourier
  integral operators also depends continuously or smoothly on the parameters.
\end{remark}

We start with a more general result for the automorphisms of the algebra of
smoothing operators. This is a form of the `Eidelheit Lemma' from
\cite{DS}. Since the setting is slightly different we give a proof.

\begin{proposition}\label{RW.62} For any two compact manifolds $M_1$ and $M_2,$
and complex vector bundles $V_1$ and $V_2$ over them, every linear
algebra isomorphism $\Psi^{-\infty}(M_1;V_1)\longrightarrow
\Psi^{-\infty}(M_2;V_2)$ is of the form
\begin{equation}
\Psi^{-\infty}(M_1;V_1)\ni A\longrightarrow GAG^{-1}\in\Psi^{-\infty}(M_2;V_2)
\label{Duistermaat-Singer.2a}\end{equation}
where $G:\CI(M_1;V_1)\longrightarrow \CI(M_2;V_2)$ is a topological
isomorphism (with respect to the standard Fr\'echet topology) with formal transpose
$G^t:\CI(M_2;V_2\otimes\Omega)\longrightarrow \CI(M_1;V_1\otimes\Omega)$
which is also a topological isomorphism in the same sense.
\end{proposition}

\begin{proof} Consider such an isomorphism between the algebras of smoothing operators
\begin{equation}
\begin{gathered}
L:\Psi^{-\infty}(M_1;V_1)\overset{\simeq}\longrightarrow \Psi^{-\infty}(M_2;V_2)\\
L(A\circ B)=L(A)\circ L(B).
\end{gathered}
\label{RW.15}\end{equation}

First, note (this is essentially Eidelheit's argument from
\cite{Eidelheit1}) that the elements, $R,$ of rank 1 in
$\Psi^{-\infty}(M;V)$ are characterized algebraically by the condition that
for any other element $A$ the composite $(AR)^2=c(AR)$ for some constant
$c.$ In one direction this is just the observation that $RAR=cR.$
Conversely if $R$ has rank two or greater it is straightforward to construct a finite
rank smoothing operator $A$ which does not have this property.

Thus $L$ must map the elements of rank 1 in the domain onto the
corresponding set in the range space. As in \S1, choose Hermitian
inner products on the bundles $V_1$ and $V_2$ and positive smooth
densities, $\nu_i,$ on each of the manifolds. The inner products induce
conjugate linear isomorphisms with the duals, $V_i'\longrightarrow V_i,$
and this allows the smoothing operators to be
identified with smooth sections of the homomorphism bundle but acting
through the antilinear isomorphism (fixed by the inner products) from
$V_i$ to $V_i'$ which will be denoted by replacing $\phi\in\CI(M_i;V_i)$ by
$\phi'\in\CI(M_i,V'_i):$  
\begin{equation}
\begin{gathered}
\Psi^{-\infty}(M_i,V_i)\equiv \CI(M_i^2;V_i\boxtimes V_i),\\
A\phi(x)=\int_{M_i}\langle a(x,\cdot),\phi'(\cdot)\rangle_{V_i}\nu_i. 
\end{gathered}
\label{RW.63}\end{equation}

Now choose a non-zero element $v\in\CI(M_1;V_1)$ with $\int_{M_1}\langle
v,v\rangle _{V_1}\nu_1=1.$ This defines a projection of rank 1, $\pi_v$
with kernel under the identification \eqref{RW.63} 
\begin{equation}
p(x,y)=v(x)v(y)=(v\otimes v)(x,y)\in\CI(M_1^2;V_1\otimes V_1).
\label{RW.64}\end{equation}
The image of this projection under $L$ is also a projection of rank
one. Choose an element $w\in\CI(M_2;V_2)$ in the range of $L(\pi_v);$ it
follows that there exists some element $h\in\CI(M_2;V_2)$ with
$\int_{M_2}\langle w,h\rangle _{V_2}\nu_2=1$ such that
\begin{equation}
L(\pi_v)\text{ has kernel } w\otimes h\in\CI(M_2^2;V_2\otimes V_2).
\label{RW.65}\end{equation}

Now, consider the subset of $\Psi^{-\infty}(M_1;V_1)$ consisting of the rank one
elements, $R,$ such that $R\pi_v=R.$ These have kernels of the form 
\begin{equation}
r=\phi\otimes v
\label{RW.66}\end{equation}
and $L(R)$ has rank one and satisfies $L(R)L(\pi_v)=L(R)$ so has kernel of
the form 
\begin{equation}
\psi\otimes h
\label{RW.67}\end{equation}
where $\psi\in\CI(M_2;V_2)$ is uniquely, and hence linearly, determined by
$\phi.$ This fixes a linear isomorphism 
\begin{equation}
G:\CI(M_1;V_1)\longrightarrow \CI(M_2;V_2)\Mst L(\phi\otimes
v)=(G\phi)\otimes h
\label{RW.68}\end{equation}
in terms of kernels.

This argument can be repeated for the operators of rank one,
$S\in\Psi^{-\infty}(M_1;V_1)$ such that $\pi_v S=S.$ This induces a second
linear (algebraic) isomorphism $H:\CI(M_1;V_1)\longrightarrow \CI(M_2;V_2)$
such that 
\begin{equation}
L(v\otimes\psi)=w\otimes (H\psi).
\label{RW.69}\end{equation}
Now, the composites of these two classes of operators are multiples of the
projections: 
\begin{equation}
SR=c\pi_v,\ S=v\otimes\eta,\ R=\phi\otimes v\Longrightarrow
c=\int_{M_1}\langle \eta,\phi\rangle _{V_1}.
\label{RW.70}\end{equation}
Since $L(SR)=L(S)L(R)$ it follows that 
\begin{equation}
\begin{gathered}
\int_{M_1}\langle \eta,\phi\rangle _{V_1}=\int_{M_2}\langle H\eta,G\phi\rangle _{V_2},\\
\int_{M_1}\langle H^{-1}\psi,\phi\rangle _{V_1}=\int_{M_2}\langle \psi,G\phi\rangle _{V_2}
\end{gathered}
\label{RW.71}\end{equation}
for all smooth sections. This shows that $G$ is a closed operator and
hence, is continuous. Moreover, $H^{-1}$ is the adjoint of $G.$

Taking the composite the other way gives a general smoothing operator of
rank one: 
\begin{equation}
RS\text{ has kernel }\phi\otimes\eta\Longrightarrow L(R)\text{ has kernel }
G\phi\otimes H\eta=(G\otimes H)\phi\otimes\eta
\label{RW.72}\end{equation}
where $G\otimes H:\CI(M_1^2;V_1\boxtimes V_1)\longrightarrow
\CI(M_2^2;V_2\boxtimes V_2).$ This is the formula
\eqref{Duistermaat-Singer.2a} on the elements of rank one, and hence its
linear span. Now the formula \eqref{Duistermaat-Singer.2a} follows in
general, since a smoothing operator is determined by its composites with
smoothing operators of rank one and these composites are themselves of rank one, so
\begin{equation}
\begin{gathered}
L(A\circ R)=L(A)\circ L(R)=GA\circ RG^{-1}=GAG^{-1}GRG^{-1}\\
\Longrightarrow L(A)=GAG^{-1}.
\end{gathered}
\label{RW.75}\end{equation}
\end{proof}

\begin{proof}[Proof of Theorem~\ref{thm:DS}] The assumption that the
  isomorphism is order-preserving implies that it induces an isomorphism
  between the smoothing ideals, so Proposition~\ref{RW.62} applies directly
  and gives $G$ for which \eqref{Duistermaat-Singer.2} holds, since as in
  the proof above a pseudodifferential operator is determined by its
  composites with rank one smoothing operators.

Denoting this isomorphism now by $F,$ it remains to show that it is a
Fourier integral operator. The notation for kernels will be continued from
the discussion above, corresponding to a choice of inner products on the
bundles and smooth positive-definite densities on the manifolds.

Consider $A\in\Psi^1(M_1;V_1)$ which is invertible with inverse
$A^{-1}\in\Psi^{-1}(M_1;V_1)$ and hence in particular is elliptic; such an
operator always exists. Let the image be $A'=FAF^{-1}\in\Psi^1(M_2;V_2);$ it
must similarly be invertible and hence elliptic. If the Schwartz kernel of $F$ is
again denoted $F\in\CmI(M_2\times M_1;V_2\otimes V_1),$ then 
\begin{equation}
FA=A'F\Longrightarrow BF=0,\ B=A'\otimes\Id-\Id\otimes A^*
\label{11.11.2009.2}\end{equation}

Although \eqref{11.11.2009.2} is not a pseudodifferential equation,
since $B$ is not in general a pseudodifferential operator on $M_2\times M_1,$ it
behaves as though it were in terms of regularity.

\begin{proposition}\label{Duistermaat-Singer.8} The Schwartz kernel of $F$
satisfies 
\begin{equation}
\begin{gathered}
\WF(F)\cap ((T^*M_2\setminus\{0\})\times\{0\})=\emptyset=\WF(F)\cap
  (\{0\}\times (T^*M_1\setminus\{0\})),\\
\WF(F)\subset\{\sigma (A)(x,-\xi)=\sigma (A')(y,\eta)\}\subset
(T^*M_2\setminus\{0\})\times(T^*M_1\setminus\{0\}).
\end{gathered}
\label{11.11.2009.3}\end{equation}
\end{proposition}

\begin{proof} The operator $B$ in \eqref{Duistermaat-Singer.2} is
  microlocally a pseudodifferential operator away from the zero sections in
  each factor of the product. The second part of \eqref{11.11.2009.3}
  corresponds to elliptic regularity in this region, meaning that the
  wavefront set of a solution $F$ to $BF=0$ is contained in the
  characteristic variety. So, it is actually the first part of
  \eqref{Duistermaat-Singer.3} that is not quite obvious. Fortunately we
  are able to choose either $A$ or $A'$ to be a differential operator, say
  of order $2$ since we know that such a globally elliptic operator, with
  positive diagonal symbol, does exist, indeed by adding large positive
  constant to it we can assume it is invertible. The corresponding
  transformed operator will not in general be differential but must be
  invertible, with inverse of order $-2$ and hence must be elliptic. It
  follows not only that the corresponding $B$ in
  \eqref{Duistermaat-Singer.3} is microlocally a pseudodifferential
  operator of order $2$ away from the zero section of $T^*M_2,$ and
  \emph{is} a pseudodifferential operator near the zero section of $T^*M_1$
  but also that it is elliptic near the zero section of $T^*M_1.$ So the
  second half of the first part of \eqref{Duistermaat-Singer.3} again
  follows by (microlocal) elliptic regularity. Reversing the roles of $M_1$
  and $M_2$ the same argument applies to give microlocality near the zero
  section of $T^*M_2.$ \end{proof}

Note in particular that the disjointness of the wavefront set of the
  kernel from the zero section in either factor implies that if
  $A\in\Psi^{-\infty}(M_1;V_1)$ is a smoothing operator, or
  $A'\in\Psi^{-\infty}(M_2;V_2),$ then separately
\begin{equation}
FA,\ A'F\text{ have kernels in }\CI(M_2\times M_1;V_2\boxtimes V_1'\otimes\Omega (M_1))
\label{Duistermaat-Singer.3}\end{equation}
-- rather than just the difference being smoothing which follows directly
from the algebra condition.

Now, the earlier parts of the argument in \cite{DS}, which we briefly recall, can
  be followed. Since, for any manifold and vector bundle,
  $\Psi^0(M;V)/\Psi^{-1}(M;V)=\CI(S^*M;\hom V)$ is an isomorphism 
  of algebras, the isomorphism $L$ must induce an algebra isomorphism of
$l$ of $\CI(S^*M_1;\hom V_1)$ to $\CI(S^*M_2;\hom V_2).$ This in turn must
  map maximal and prime ideals to such ideals. These ideals in $\CI(S^*M;\hom V)$
  correspond to vanishing at a point and hence $l$ induces a bijection
  $\chi:S^*M\longrightarrow S^*M'$ and $l$ must itself be pull-back with
  respect to the inverse of $l.$ Thus $\chi$ must be a diffeomorphism such 
  that the second part of \eqref{11.11.2009.3} is refined to 
\begin{equation}
\WF(F)\subset\operatorname{graph}'(\chi),
\label{11.11.2009.4}\end{equation}
the twisted graph. The same argument of course applies to the inverse of
$G$ and the inverse of $\chi.$ 

The quotient $\Psi^1(M;V)/\Psi^0(M;V)$ is the space of smooth sections
  on $T^*M\setminus 0_M\longrightarrow \pi^*\hom(V)$ which are homogeneous
  of degree $1.$ It is generated, modulo the multiplicative action of
  $\CI(S^*M;\hom V),$ by the symbol of one element with scalar principal
  positive and elliptic symbol. In  \cite{DS} it is observed that the 
  symbol of the image must be real and non-vanishing. In fact it follows easily
  that it must be positive, at least for compact manifolds. Indeed, $A+\tau\Id$
  is invertible for all $\tau\in\bbC\setminus(-\infty,0)$ so the same must be
  true of $A'.$ If its symbol were negative then $A'=Q_0-Q_1$ where $Q_1$
  is positive (so self-adjoint). It follows that $P'-t$ has an inverse in
  $\Psi^{-1}(M')$ for $t>T,$ so the same is true of $A$ which contradicts
  the spectral theorem. Thus in fact the image of the symbol of $A$ must be
  positive.

Again as in \cite{DS} it follows that the symbols of $A'$ and $A$ for
all orders are related by a homogeneous diffeomorphism which projects to
$\chi.$ The behaviour of commutators shows that this diffeomorphism must be
symplectic for the difference symplectic structure, that is
`canonical'.

Lemma 2 in \cite{DS} now applies without the possibility of the
`anticanonical' map. So, there is a homogeneous canonical transformation,
still denoted $\chi,$ such that under the isomorphism $\sigma
(A)=\chi^*\sigma (A')$ and $\WF(F)\subset\operatorname{graph}'(\chi).$

The remaining steps are now slightly changed in that we are trying to
  show that the given operator $F$ is a Fourier integral operator
  associated to $\chi.$ We can now work with the
  weaker intertwining condition that for each $A\in\Psi^{k}(M)$ there
  exists $A'\in\Psi^k(M)$ such that the kernels satisfy
\begin{equation}
FA=A'F+\CI(M'\times M;V\otimes V'\otimes\Omega _R),
\label{11.11.2009.5}\end{equation}
that is, they are conjugate up to smoothing errors.

Choose a global pair of elliptic Fourier integral operators, $G$ and $H,$
associated to $\chi$ and $\chi^{-1}$ not necessarily invertible but
essential inverses of each other. Consider $\tilde F=HF$ acting now from
$\CI(M_1;V_1)$ to $\CI(M_1;V_1).$ From the calculus of wavefront sets this has operator
wavefront set in the identity relation, the same as a pseudodifferential
operator and we wish to show that it is one. From \eqref{11.11.2009.5} we
deduce that for each $A\in\Psi^k(M)$ there exists $I(A)\in\Psi^k(M),$
namely $I(A)=HA'G,$ such that
\begin{equation}
\tilde FA=I(A)\tilde F+A'',\ A''\in\Psi^{-\infty}(M),\ E(A)=I(A)-A\in\Psi^{k-1}(M).
\label{Duistermaat-Singer.4}\end{equation}

Again we proceed essentially as in \cite{DS}, see also
\S\ref{sect:der}. The symbol of $E(A)$ is given by a derivation, hence a
vector field, homogeneous of degree $-1,$ applied to the symbol of $A$ and
which distributes over the Poisson bracket:
\begin{equation}
\sigma _{k-1}(E(A))=V\sigma _k(A),\ V\{a,b\}=\{Va,b\}+\{a,Vb\}.
\label{Duistermaat-Singer.5}\end{equation}
From this it follows that it is locally Hamiltonian on $T^*M\setminus0_M,$
meaning that
\begin{equation}
\omega(V,\cdot)=\gamma 
\label{Duistermaat-Singer.6}\end{equation}
where $\gamma$ is a well-defined smooth closed 1-form which is homogeneous of
degree $0$ on $T^*M\setminus0_M.$ The only such 1-forms are locally (on
cones in $T^*M\setminus0_M)$ the differentials of the sum of a smooth
function homogeneous of degree $0$ and $s\log|\xi|$ where $|\xi|$ is a real
positive function homogeneous of degree $0$ and $s$ is a complex constant. 

Now working microlocally, and iterating the argument over orders as
  in \cite{DS}, we can construct an elliptic pseudodifferential
  operator $D\in\Psi^s(M)$ which, by formal conjugation, gives the same
  relation with $I(A)$ as in \eqref{Duistermaat-Singer.4} for all $A$ with
  essential support near the given point. It follows that the
  corresponding point on the diagonal is not in $\WF'(\tilde F-D).$ This however
  proves that $\tilde F$ is \emph{globally} a pseudodifferential operator, since
  it is a well-defined operator with the correct wavefront set relation and
  is microlocally everywhere a pseudodifferential operator. Hence $F$ is a
  Fourier integral operator as claimed.
\end{proof}


\section{Group of invertible Fourier integral operators}\label{Sect:GFIO}


Let us consider in more detail the group, $\GL(\cF^0(Z;V)),$ of invertible
Fourier integral operators of order $0$ on a fixed complex vector bundle
$V$ over a compact manifold $Z.$ The topology on the symbolic quotient of
the algebra of Fourier integral operators is discussed by Adams et al in
\cite{Adams}. The group of Fourier integral operators is shown to be a
Fr\'echet group by Omori in \cite{Omori} and in papers cited there with
Kobayashi, Maeda and Yoshioka. It is useful (although not really exploited
in the present paper for which the results in \cite{Omori} suffice) to have
somewhat stronger results for the group of pseudodifferential operators so we proceed
to discuss the topology in some detail.

We start by examining the group of invertible classical pseudodifferential
operators of order $0,$ denoted $G^0(Z;V)\subset\Psi^0(Z;V).$ This is an
open subset which is a Fr\'echet Lie group, but much more is true. Namely
we show that there is a decreasing sequence of Banach algebras such that
$G^0(Z;V)$ is the projective limit of the corresponding sequence of the
groups of invertibles. The basis of this is the characterization of
pseudodifferential operators by commutation conditions due to R. Beals,
\cite{Beals, Beals77, Beals79}.  Beals was interested in showing that pseudodifferential
operators could be characterized in terms of boundedness properties of
commutators and composites with differential operators. Here we are only
trying to obtain the natural topology on the (group of invertible)
pseudodifferential operators, and then Fourier integral operators, so we
are free to simplify the discussion by using pseudodifferential operators
in the defining properties.

We proceed to define, inductively, a sequence of linear subspaces 
\begin{equation}
\cB_j\subset\cB_{j-1}\dots\subset\cB_0.
\label{27.August.2012.33}\end{equation}
Namely, $A\in\cB_j,$ for $j\in\bbZ,$ provided the following conditions hold.
\begin{enumerate}
\item $A\in\cB_{j-1}.$
\item $A$ restricts to the first Sobolev space, $A(H^1(Z;V))\subset
  H^1(Z;V)$ and if $f\in\CI(Z)$ and $D\in\Psi^1(Z;V)$ then $D[f,A],$
  $[f,A]D\in\cB_{j-1},$ by continuous extension in the first case.
\item If $D\in\Psi^1(Z;V)$ is such that $\sigma _1(D)$ is a scalar symbol,
then $[D,A]\in\cB_{j-1}.$
\item If $D\in\Psi^1(Z;V),$ $E\in\Psi^{-1}(Z;V)$ then $DAE,$
  $EAD\in\cB_{j-1},$ again by continuous extension in the second case.
\end{enumerate}

Then we set $\cG_j=\cB_j\cap \cG_0.$ Although a further refinement is needed
to yield the classical pseudodifferential operators, we first analyze the properties
of these $\cB_j$ and $\cG_j.$ 

\begin{proposition}\label{22b.9.2012.1} Each $\cB_j$ is a Banach algebra
  and $\cG_j$ is a smooth Banach Lie group which is the open subset of
  invertible elements in $\cB_j$ for which the product and inverse maps are
  smooth with respect to the Banach manifold topology.
\end{proposition}

\begin{proof} Certainly $\cB_0$ is a Banach space and $\cG_0$ is a Banach Lie
  group which is the open subset of invertible elements. So it suffices to
  proceed by induction.

By standard properties of pseudodifferential operators, 
\begin{equation}
\Psi^0(Z;V)\subset\cB_j\ \forall\ j.
\label{22b.9.2012.3}\end{equation}
Indeed, all the elements in the defining conditions are then in
$\Psi^0(Z;V)$ so the statement follows by induction. From this it follows
that each $\cB_j$ is a right and left module over $\Psi^0(Z;V)$ which is
used without further comment below.

To simplify the defining conditions of $\cB_j,$ let $Q\in\Psi^1(Z;V)$ be 
elliptic and invertible with diagonal principal symbol.

First we check that $\cB_j$ defined by these conditions in terms of some $\cB_{j-1}$ is
  an algebra if $\cB_{j-1}$ is an algebra. So suppose $A_1,$
  $A_2\in\cB_j\subset \cB_{j-1}$ and consider condition (2). Then
  $A_1A_2\in\cB_{j-1},$ it must restrict to map the Sobolev space
  $H^1(Z;V)$ into itself and
\begin{equation}
D[A_1A_2,f]
=DA_1[A_2,f]+D[A_1,f]A_2=DA_1Q^{-1}Q[A_2,f]+D[A_1,f]A_2\in\cB_{j-1}
\label{22b.9.2012.2}\end{equation}
using (2) and (3) for the factors. Similarly for $[A_1A_2,f]D,$ so (2)
holds for the product. Condition (3) is even simpler and (4) holds for
$A_1A_2$ since 
\begin{equation}
DA_1A_2E=(DA_1Q^{-1})(QA_2E),\ EA_1A_2Q=(EA_1Q)(Q^{-1}A_2D).
\label{22b.9.2012.4}\end{equation}
It follows from an inductive application of this argument that all the
$\cB_j$ are algebras.

At this stage, it is convenient to refine the second condition. Note that
if $f,$ $g\in\CI(M)$ have disjoint supports then $DfAg=D[f,A]g.$ So it
follows from the second condition that
\begin{equation}
fAgD,\ DfAg\in\cB_{j-1}\Mif f,\ g\in\CI(M),\ \supp(f)\cap\supp(g)=\emptyset.
\label{22b.9.2012.6}\end{equation}

To exploit this, choose a finite covering of the manifold by coordinate
patches, $U_a\subset M$ each diffeomorphic to a ball centred at the origin
and with a fixed coordinate system $x_{k,a}\in\CI(U_a),$ $1\le k\le\dim M.$
Then choose a partition of unity $\chi_a\in\CI(M)$ subordinate to the $U_a$
and a second collection of functions $\chi'_a\in\CIc(U_a)\subset\CI(M)$
such that $\chi'_a\equiv1$ in a neighbourhood of $\supp(\chi_a)$ for each
$a.$ This allows any operator $A\in\cB_l$ to be decomposed as
\begin{equation}
A=\sum\limits_{a}A_a+\sum\limits_{a}\chi_aA(1-\chi'_a),\ A_a=\chi_aA\chi'_a,
\label{22b.9.2012.7}\end{equation}
where all terms are in $\cB_l.$ Thus the second condition above defining
$A\in\cB_j,$ where we know that $A\in\cB_{j-1},$ certainly implies that 
\begin{equation}
\begin{gathered}
Q\chi_aA(1-\chi'_a),\ \chi_aA(1-\chi'_a)Q\in\cB_{j-1}\ \forall\ a,\\
Q[x_{k,a},A_a],\ [x_{k,a},A_a]Q\in\cB_{j-1}\ \forall\ a,k.
\end{gathered}
\label{22b.9.2012.8}\end{equation}
The second collection of conditions makes sense since supports are confined
to $U_a.$

Conversely, this finite collection of conditions implies the
second condition above for all $f\in\CI(M).$ Again we prove this by
induction, which in particular will show that we may define a norm
$\|\cdot\|_{j}$ on $\cB_j$ by setting 
\begin{multline}
\|A\|_j=\|A\|_{j-1}+\sum\limits_{a}\|Q\chi_aA(1-\chi'_a)\|_{j-1}+
\sum\limits_{a}\|\chi_aA(1-\chi'_a)Q\|_{j-1}\\
+\sum\limits_{a,k}\|Q[x_{k,a},A_a]\|_{j-1}+
\sum\limits_{a,k}\|Q[x_{k,a},A_a]\|_{j-1}\\
+\|[Q,A]\|_{j-1}+\|QAQ^{-1}\|_{j-1}+\|Q^{-1}AQ\|_{j-1}.
\label{27.August.2012.38}\end{multline}

\begin{lemma}\label{27.August.2012.39} The conditions \eqref{22b.9.2012.8}
  on $A\in\cB_{j-1},$  
  together with the requirements $[Q,A],$ $QAQ^{-1}$ and
  $Q^{-1}AQ\in\cB_{j-1}$ imply that $A\in\cB_j$ and
  \eqref{27.August.2012.38} (defined inductively) is a norm on $\cB_j$ with
  respect to which it is a Banach algebra, so $\|A_1A_2\|_{j}\le
  C_j\|A_1\|_{j}\|A_2\|_{j}$ for all $A_1,$ $A_2\in\cB_j.$
\end{lemma}

\begin{proof} Proceeding by induction we may suppose the result known for
  $\cB_l$ for $l\le j-1.$ Expanding $DfA$ using
\eqref{22b.9.2012.7} the terms arising from the second sum are already in
$\cB_{j-1}$ by the first part of \eqref{22b.9.2012.8} and similarly for
$DAf.$ Thus it is enough to consider the commutators with each of the localized
operators $A_a.$ That is we need to check that \eqref{22b.9.2012.8} implies that
\begin{equation}
Q[f,A_a]\in\cB_{j-1}\ \forall\ f\in\CIc(U_a).
\label{27.August.2012.34}\end{equation}

Consider the commutator with the oscillating exponential in the local
coordinate patch $U_a,$ where the local coordinates in $U_a$ are now denoted
simply by $x:$
\begin{equation}
F(\xi)=Q[e^{ix\cdot\xi},A_a],\ \xi\in\bbR^n,\ n=\dim M.
\label{27.August.2012.35}\end{equation}
By the third assumption, this is a bounded operator from $H^1(Z;V)$ to
$L^2(Z;V)$ and as such is clearly smooth and
\begin{equation}
\frac{d}{ds}F(s\xi)=Q[ix\cdot\xi e^{isx\cdot\xi},A_a]=
Qix\cdot\xi Q^{-1}F(s\xi)+\sum\limits_{k}i\xi_kQ[x_k,A_a]e^{isx\cdot\xi}.
\label{27.August.2012.36}\end{equation}
Since $F(0)=0$ solving this differential equation gives 
\begin{equation}
F(\xi)=e^{iQx\cdot\xi Q^{-1}}
\int_0^1\left(e^{-isQx\cdot\xi Q^{-1}}\sum\limits_{k}i\xi_kQ[x_k,A_a]e^{isx\cdot\xi}\right)ds
\label{27.August.2012.37}\end{equation}

Now, by hypothesis, each of the $Q[x_k,A_a]\in\cB_{j-1}$ and it is
straightforward to check that $e^{isx\cdot\xi}$ and $e^{-isQx\cdot\xi
  Q^{-1}}$ for $s\in[-1,1]$ are continuous maps into $\cB_{j-1}$ with norms which
grow at most polynomially in $|\xi|.$ Thus $\|F(\xi)\|_{j-1}\le
C(1+|\xi|)^{N(j)}.$ Now, expressing $f\in\CIc(U_a)$ in terms of its Fourier
transform, which is a rapidly decreasing function of $\xi,$ it follows by
integration that $Q[f,A_a]\in\cB_{j-1}.$ A similar argument applies to
$[f,A_a]Q$ so the first part of the Lemma follows.

Completeness of $\cB_j$ also follow inductively and the earlier argument
that $\cB_j$ is an algebra extends directly to give the product estimate on
the norm. 
\end{proof}

Note that it is always possible to rescale $\|\cdot\|_j$ by the constant
$C_j$ so that  
\begin{equation}
\|AB\|_j\le\|A\|_j\|B\|_j,\ \forall\ A,B\in\cB_j.
\label{27.August.2012.40}\end{equation}

Returning to the proof of the Proposition, we now show that 
\begin{equation}
\bigcap_j\cB_j=\Psi^0_{\infty}(Z;V)
\label{27.August.2012.41}\end{equation}
is the larger algebra of pseudodifferential operators `with symbols
relative to $L^\infty$' but not necessarily classical. This is essentially
the content of Beals' result, so we only briefly recall the argument.

So, suppose $A\in\cB_j$ for all $j.$ We know that the conditions
\eqref{22b.9.2012.6} apply inductively to show that $Q^pfAg$ and $fAgQ^p$
are bounded on $L^2(Z;V)$ for all $p$ if $f,$ $g\in\CI(Z)$ have disjoint
supports. This simply means that the Schwartz kernel of $A$ is smooth away
from the diagonal. Thus it suffices to show that each of the $A_a=\chi
A\chi'_a$ is a pseudodifferential operator in the local coordinate chart
$x_{k,a}$ which we can denote by $x_k.$ Since the kernel of $A_a$ now has
compact support in the coordinate patch it can be written in oscillatory
integral form  
\begin{equation*}
A_a=(2\pi)^{-n}\int e^{i(x-x')\cdot\xi}e(x,\xi)d\xi
\label{27.August.2012.42}\end{equation*}
where $e$ is smooth in $\xi$ and possibly a distribution in $x.$ The
commutation conditions show that the operator obtained by replacing
$e(x,\xi)$ by 
\begin{equation}
\xi^\gamma D_x^\alpha D_\xi^\beta e(x,\xi),\ |\gamma|\le|\beta|
\label{27.August.2012.43}\end{equation}
is also bounded on $L^2.$ Now following \cite{Beals} this shows
that $A_a$ is a pseudodifferential operator of some fixed order in the
sense that 
\begin{equation}
|D_x^\alpha D_\xi^\beta e(x,\xi)|\le C_{\alpha ,\beta }(1+|\xi|)^{N-\beta }
\label{27.August.2012.44}\end{equation}
but then boundedness on $L^2$ implies that \eqref{27.August.2012.44} holds with $N\le 0.$
This completes the argument in the case of `pseudodifferential operators
with bounds'.

To refine the argument to give \emph{classical} pseudodifferential
operators we need to add an extra condition to the iterative definition of
$\cB_j\subset\cB_{j-1}.$ Namely in the local expression
\eqref{27.August.2012.42} for $A_a$ we need to ensure that $e(x,\xi)$ has
an asymptotic expansion in terms of decreasing integral homogeneity. This
follows from the same type of estimates \eqref{27.August.2012.44} with
$N=0$ on the iteratively differentiated amplitude:-
\begin{equation}
|D_x^\alpha
D_\xi^\beta\left(q(x,\xi)^L\left(\prod_{p=0}^L(\xi\cdot\pa_{\xi}+p)\right)e(x,\xi)\right)|\le
C_{\alpha ,\beta }(1+|\xi|)^{-|\beta|}
\label{27.August.2012.45}\end{equation}
for all $L,$ where $q$ is an elliptic symbol.

This can be arranged by appending the conditions 
\begin{equation}
\cR(A)=Q\sum\limits_{a}\sum\limits_{k}D_{x_{k,a}}[x_{j,a},A_a]\in\cB_{j-1}.
\label{27.August.2012.46}\end{equation}
Again these conditions are independent of the partition of unity
(subordinate to a coordinate cover) used. 

\begin{lemma}\label{27.August.2012.49} For any partition of unity
  subordinate to a coordinate covering, \eqref{27.August.2012.46} gives a
  map 
\begin{equation}
\cR:\Psi^0(Z;V)\longmapsto \Psi^0(Z;V).
\label{27.August.2012.50}\end{equation}
Conversely, if $A\in\Psi^0_{\infty}(Z;V)$ is a pseudodifferential operator
`with bounds' then for any $l\in\bbN,$
\begin{equation}
\cR^l(A)\in\Psi^0_{\infty}(Z;V)\Longrightarrow
A=A'+A'',\ A'\in\Psi^0(Z;V),\ A''\in\Psi^{-l}_{\infty}(Z;V).
\label{27.August.2012.51}\end{equation}
\end{lemma}

\begin{proof} From the basic properties of pseudodifferential operators 
\begin{equation*}
\cR(A):\Psi^0_{\infty}(Z;V)\longrightarrow \Psi^1(Z;V).
\label{27.August.2012.56}\end{equation*}
Moreover, 
\begin{equation}
\sigma _1(\cR(A))=\frac1 iq(x,\xi)\sum\limits_{k} \xi_k\pa_{\xi_k}\sigma
_0(A)=\frac1i qR\sigma _0(A)
\label{27.August.2012.52}\end{equation}
is given by the radial vector field on $T^*Z$ applied to the symbol (which
is well-defined even though the action is on a vector bundle since the
bundle is lifted from $Z).$ Thus, if $A$ is classical then $\sigma _0(A)$
is represented by a function which is homogeneous of degree $0$ and hence
$\sigma _1(\cR(A))=0,$ modulo symbols of order $-1,$ and
\eqref{27.August.2012.50} follows.

Conversely, if $A\in\Psi^0_{\infty}(Z;V)$ and
$\cR(A)\in\Psi^0_{\infty}(Z;V)$ then, since $Q$ is elliptic, $R\sigma
_0(A)$ is a symbol of order $-1$ and \eqref{27.August.2012.51} holds for
$l=1$ by radial integration. Proceeding inductively, so assuming
\eqref{27.August.2012.51} for $l\le p-1,$
\begin{equation}
\cR^p(A)\in\Psi^0_{\infty}(Z;V)\Longrightarrow
\cR(A)=B'+B'',\ B'\in\Psi^0(Z;V),\ B''\in\Psi^{-p+1}_{\infty}(Z;V).
\label{27.August.2012.53}\end{equation}
From the first part of the Lemma, $\cR^{p}(B')\in\Psi^0(Z;V)$ so 
\begin{equation}
\cR^{p-1}(B'')\in\Psi^0_{\infty}(Z;V).
\label{27.August.2012.54}\end{equation}
Now, since $B''$ is of order at most $-p+1$ it follows directly that
$\cR^p(B'')$ is of order $1$ and its principal symbol can be computed
directly in terms of the radial vector field $R$ on $T^*Z$ and the
principal symbol $q$ of $Q:$
\begin{equation}
i^{-p}(qR)^pb''=i^{-k}q^p(R+p-1)(R+p-2)\cdots Rb'',\ b''=\sigma _{-p+1}(B'').
\label{27.August.2012.55}\end{equation}
Thus the iterative condition implies that the leading symbol of $B''$ is
homogeneous of degree $-p+1,$ modulo symbols of order $-p,$ and the
inductive hypothesis follows for $l=p.$ This completes the proof of
\eqref{27.August.2012.51}. 
\end{proof}

Finally then it follows that the classical algebra $\Psi^0(Z;V)$ is the
projective limit of the Banach algebras obtained by appending
\eqref{27.August.2012.46} to the inductive definition of $\cB_j$ and adding
corresponding terms to the norm $\|\cdot\|_j.$

This completes the characterization of the Fr\'echet topology on
$\Psi^0(Z;V)$ and shows that $G^0(Z;V)$ is indeed the intersection of the
group of invertible elements.
\end{proof}

As noted above, this characterization of $G^0(Z;V)$ as a projective limit
of smooth Banach groups is stronger than the earlier descriptions in the
literature. The well-known fact that the Lie algebra of $G^0(Z;V)$ is
$\Psi^0(Z;V)$ follows easily. It is also the case that the exponential map
from the Lie algebra $\Psi^0(Z;V)$ to $G^0(Z;V)$ is a smooth isomorphism of
a neighbourhood of $0$ to a neighbourhood of the identity. This can be seen
from the holomorphic functional calculus. Namely, if $B(\Id,\frac14)$ is the
ball, in terms of the $L^2$ operator norm, around the identity in
$G^0(Z;V)$ then for each $A\in B(\Id,\frac14),$ $\log(A)\in\cB_0$ can be
represented in terms of a contour integral around $|z-1|=\ha$ of
the resolvent family. Since the latter is necessarily a map into
$\Psi^0(Z;V)$ it follows directly that $\log A\in\Psi^0(Z;V)$ and that
$\exp(\log A)=A.$ Thus the group $G^0(Z;V)$ behaves 
in many ways like a Banach Lie group.

Now, we pass to the more complicated group $\GL(\cF^0(Z;V))$ of invertible
Fourier integral operators. First consider the subgroup $\GL_{\dag}(\cF^0(Z;V))$
corresponding to canonical transformations which are in the connected
component of the identity, $\Can_0(Z).$ This subgroup gives a 
short exact sequence of topological groups, which is shown in Lemma~\ref{27.August.2012.48}
to be a principal bundle,
\begin{equation}
\xymatrix{
G^0(Z;V)\ar@{-}[r]&\GL_{\dag}(\cF^0(Z;V))\ar[d]\\
&\Can_0(Z).
}
\label{27.August.2012.47}\end{equation}
Omori in \cite{Omori}, see also \cite{Michor}, discusses the diffeomorphism
group of a compact manifold and this discussion applies to the contact
group to give a `projective (inverse) limit Hilbert' structure on
$\Can_0(Z)$ arising from the completion of the group with respect to
Sobolev topologies (of order tending to infinity).

\begin{lemma}\label{27.August.2012.48} The group $\GL_{\dag}(\cF^0(Z;V)),$
  corresponding to \eqref{27.August.2012.47}, in the natural Fr\'echet topology on
  the kernels as Lagrangian distributions, is a principal $G^0(Z;V)$ bundle
  as in \eqref{27.August.2012.47} over $\Can_0(Z).$
\end{lemma}

\begin{proof} As in \cite{Adams} it is straightforward to construct a
  section of \eqref{27.August.2012.47} near the identity in the
  base. Multiplication and inversion are contunuous (in fact smooth) in the
  resulting Fr\'echet topology near the identity and this allows the
  topolgy to be extended to the rest of $\GL_{\dag}(\cF^0(Z;V))$ -- see also
  Appendix~A.
\end{proof}

The topology then extends to other components (of $\Can(Z))$ in the
same way. Note that in general it is not clear that the map from
$\GL(\cF^0(Z;V))$ to $\Can(Z)$ is surjective, since there may be an
index obstruction to the invertibility of a Fourier integral operator
corresponding to a given canonical transformation and it may not, for a
particular $V,$ be possible to cancel this obstruction by composition with
an elliptic pseudodifferential operator in $\Psi^0(Z;V)$ with the opposite
index (if the index map for pseudodifferential operators on $V$ is not surjective).

We are also interested in the projective quotient of this group
\begin{equation}
\PGL(\cF^\bbC(Z;V))=\GL(\cF^\bbC(Z;V))/\bbC^*\Id.
\label{RW.76}\end{equation}
Various of its normal subgroups and their projective images will play an
important role in the subsequent discussion of the index map.

For an elliptic operator the order is unambiguously determined so defines
an additive homomorphism giving short exact sequences 
\begin{equation}
\begin{gathered}
\xymatrix{
\GL(\cF^0(Z;V))\ar[r]\ar[dd]&\GL(\cF^\bbC(Z;V))\ar[dr]\ar[dd]\\
&&\bbC.\\
\PGL(\cF^0(Z;V))\ar[r]&\PGL(\cF^\bbC(Z;V))\ar[ur]
}
\end{gathered}
\label{RW.77}\end{equation}
Surjectivity here follows from the existence of an invertible
pseudodifferential operator of order $s$ for any $s\in\bbC.$ 

Since each operator is elliptic the Schwartz kernel determines the contact
transformation $\chi,$ since its wavefront set must be equal to the twisted
graph of the associated canonical transformation. The same is true of the
projective group, since operators identified in the quotient are all
elliptic.  Thus there are well-defined homomorphisms with the
pseudodifferential operators mapping to the identity diffeomorphism so
giving sequences
\begin{equation}
\begin{gathered}
\xymatrix{
\GL(\Psi^0(Z;V))\ar[r]\ar[dd]&\GL(\cF^0(Z;V))\ar[dr]\ar[dd]\\
&&\Con(S^*Z)\\
\PGL(\Psi^0(Z;V))\ar[r]&\PGL(\cF^0(Z;V))\ar[ur]
}
\end{gathered}
\label{RW.33}\end{equation}
where $\Con(S^*Z)$ is the group of contact diffeomorphisms of the cosphere
bundle $S^*Z$ and there is a similar diagram for the operators of general
complex order.

\begin{lemma}\label{RW.78} The horizontal sequences in \eqref{RW.33} (and
  the similar ones for general complex order) are exact provided
  $\Psi^0(Z;V)$ contains an operator of index $1.$
\end{lemma}

\begin{proof} There are certainly elliptic Fourier integral operators
  associated to any canonical transformation. Composition with an elliptic
  Fourier pseudodifferential operator of order $0$ shifts the index so if
  there is such an operator of index $1$ repeated composition with either
  this operator or its parametrix gives a Fourier integral operator of
  index $0$ which can then be perturbed by a smoothing operator to be invertible.
\end{proof}

\begin{rems}
Note that there is always an elliptic operator in   $\Psi^0(Z;V)$ of index equal to 1
when the rank of $V$ is in the stable range, $\rank(V)>\dim(Z)/2$ (see \cite{AS1}).

\end{rems}


\section{Bundles of pseudodifferential operators}\label{sect:pdobundle}


The main object of study in this paper is a bundle of algebras with typical fiber the
pseudodifferential operators acting on sections of a vector bundle over a
fixed compact manifold.

\begin{definition}\label{RW.29} A \emph{(filtered) bundle of
    pseudodifferential algebras,} $\bfPsi^\bbZ,$ over a manifold, $X,$ is a
  fiber bundle which is a Fr\'echet manifold with typical fibre the algebra
  $\Psi^{\bbZ}(Z;V)$ for some fixed compact manifold $Z$ and vector bundle
  $V.$ That is, $\bfPsi^\bbZ$ is equipped with a surjective smooth map
\begin{equation}
p:\bfPsi^\bbZ\longrightarrow X
\label{RW.22}\end{equation}
the fibres of which are $\bbZ$-filtered algebras and such that any point of
$X$ has an open neighbourhood $U$ on which there is a smooth bijection
\begin{equation} f_U:U\times \Psi^{\bbZ}(Z;V)\longrightarrow p^{-1}(U)
  \label{RW.24}\end{equation} reducing $p$ to projection onto the first
factor and which is an order-preserving isomorphism of algebras at each
point.
\end{definition}

Let $U_i$ be an open cover of $X$ by such trivializations. The assumption
of smoothness on the bundle means it has local smooth sections through any
point and it follows from the result of Duistermaat and Singer, in the form
of Theorem~\ref{thm:DS} above, see also Remark \ref{PostArx.4} following the statement,
that the transition functions
\begin{equation}
g_{ij} : U_i \cap U_j \longrightarrow \PGL(\cF^\bbC(Z;V)),
\label{RW.26}\end{equation}
are smooth maps into $\PGL(\cF^\bbC(Z;V))$ as a Fr\'echet Lie group as
discussed in Lemma~\ref{27.August.2012.48}. 

The transition functions in equation \eqref{RW.26} satisfy the cocycle condition,
\begin{equation}
g_{ij} g_{jk} g_{ki} =1, \qquad g_{ij}=g_{ji}^{-1}, \qquad g_{ii}=1.
\end{equation}
 The principal bundle, ${\bf F},$ associated to
$\bfPsi^\bbZ,$ is modelled on the Fr\'echet Lie group of projective
invertible Fourier integral operators, $\PGL(\cF^\bbC(Z;V)),$ formed
from the algebra of Fourier integral operators of complex order on $Z$ with
coefficients in the complex vector bundle $V$ discussed above.

\begin{proposition}\label{RW.25} The structure group of any filtered bundle
  of pseudodifferential algebras can be reduced to the Fr\'echet group
  $\PGL(\cF^0(Z;V))$ corresponding to invertible elements in the algebra of
  Fourier integral operators of order $0$ rather than general complex
  order.
\end{proposition}

\begin{proof} As shown originally by Seeley, the complex powers of a
  positive differential operator of second order, which always exists, form
  an entire family of pseudodifferential operators $Q^s\in\Psi^s(Z;V)$
  where $Q^2$ is the original differential operator.

Take a covering of the base $X=\bigcup_jU_j$ by open sets over each of
which the given bundle of pseudodifferential operators $\bfPsi^\bbZ$ is
trivial, with $f_i:U_i\longrightarrow p^{-1}(U_i)$ the associated
trivializations. Thus the transition maps \eqref{RW.26} can be realized by
smooth maps $U_i\cap U_j\ni x\longmapsto \tilde g_{ij}(x)\in
\cF^{s_{ij}(x)}(Z;V)$ where the $s_{ij}:U_i\cap U_j\longrightarrow \bbC$
are smooth maps which necessarily satisfy the (additive) cocycle
condition. Such a cocycle is necessarily trivial, so there exist smooth
maps $s_i:U_i\longrightarrow \bbC$ such that $s_{ij}=s_i-s_j$ on $U_i\cap
U_j;$ for instance using a partition of unity subordinate to the cover one
can take $s_i=\sum\limits_{k\not=i}\phi_k s_{ik}.$

Now, replace the trivializations $f_i$ by $f'_i(x)=f_i(x)Q^{-s_i(x)}.$ The
corresponding transition maps are
\begin{equation}
g'_{ij}=Q^{s_j(x)}f_j^{-1}(x)f_i(x)Q^{-s_i(x)}\in \cF^0(Z;V).
\label{RW.28}\end{equation}
This gives the desired reduction of the structure group.
\end{proof}

Thus a filtered bundle of pseudodifferential operators gives rise to an
isomorphism class of principal bundles with structure group
$\PGL(\cF^0(Z;V)).$ We will assume from this point onwards that some choice
of this principal bundle has been made, using Proposition~\ref{RW.25}. All
the results below are independent of this choice.

Then $\bfPsi^{\bbZ}$ is an associated bundle, corresponding to the adjoint
action and it has a well-defined smoothing subbundle
$\bfPsi^{-\infty}$ corresponding to the action on smoothing operators. Since
the elements of $\cF^0(Z;V)$ are bounded operators on the Hilbert space
$H=L^2(Z;V),$ the principal bundle can be extended to a bundle with
structure group $\PGL(H)$, and then reduced to a bundle with structure
group $\PU(H)$, since the quotient $\PGL(H)/\PU(H)$ is
contractible. Correspondingly the smoothing bundle $\bfPsi^{-\infty}$ is
then a subbundle of the associated bundle of compact operators, which is an
Azumaya bundle over $X$ in the sense of \cite{MMS1, MMS2}. The image in
$\bfH^3(X;\bbC)$ of the Dixmier-Douady class, itself in $\bfH^3(X;\bbZ),$
classifying this (completed) Azumaya bundle is computed below in
`Chern-Weil' form from the bundle.

For completeness, we include the \v Cech definition of the Dixmier-Douady invariant here.

\begin{definition}\label{def-Cech-DD}
Let $g_{ij}: U_i\cap U_j
\longrightarrow \PGL(\cF^0)$ denote the transition functions  of
${\bf F}$ with respect to some good open cover $\{U_i\}$ of $X.$ Let 
$\hat g_{ij}: U_i\cap U_j \longrightarrow \GL(\cF^0)$ denote a lift of $g_{ij}.$ 
Since $g_{ij}$ satisfies the cocycle identity,
$$
\hat g_{ij} \hat g_{jk} \hat g_{ki} = c_{ijk} \Id: U_i\cap U_j \cap U_k \longrightarrow \bbC^*
$$
is a \v Cech 2-cocycle on $X$, that is, $[c_{ijk}] \in \bfH^2(X;
\underline{\bbC^*})\cong \bfH^3(X; \bbZ).$ This is the {\em Dixmier-Douady
  invariant} of ${\bf F} $ and is independent of the trivialization (and
the reduction to order $0)$.
\end{definition}

The Dixmier-Douady invariant measures the failure of the bundle $\bfPsi^\bbZ$
to be a bundle of operators. It vanishes, as shown below, if and only if
${\bf F}$ lifts to a principal $\GL(\cF^0(Z;V))$ bundle
${\widehat{\bf F}}.$ In this case there is an associated bundle of
Fr\'echet spaces over $X,$
$$
\bCI={\widehat{\bf F}} \times_{\GL(\cF^0(Z;V))} \CI(Z;V)\longrightarrow  X,
$$
and $\bfPsi^{\bbZ}$ is a bundle of operators on $\bCI.$ In fact the associated bundle
of smoothing operators can then be realized as the smooth completion of the
tensor product
$$
\bfPsi^{-\infty} = {\widehat{\bf F}} \times_{\GL(\cF^0(Z;V))}
\CI(Z^2;V\boxtimes V'\otimes\Omega_L)\longrightarrow X,
$$
of $\bCI$ and the analogous bundle of sections of the dual bundle tensored
with the density bundle, to which the adjoint action of ${\bf F}$ extends.

The symbol of a Fourier integral operator of order $0$ (associated to a
canonical diffeomorphism) can be identified with an invertible section of
the pull back to the cosphere bundle of $\hom(V)$ tensored with the
(pull-back of) the Maslov bundle $\cL_\chi$. A non-vanishing constant multiple has the
same canonical transformation so this gives a projective symbol map which
is multiplicative under left composition with the group of projective pseudodifferential
operators
\begin{equation}
\xymatrix{
\PGL(\Psi^0(Z;V))\ar@{-}[r]\ar[d]^-{\sigma}&\PGL(\cF^0(Z;V)) \ar[d]^-{\sigma}\\
\CI(S^*Z;\Aut(V)\ar@{-}[r]&\CI(S^*Z;\pi^*{\rm Aut}(V)\otimes \cL_\chi),
}
\label{RW.31}\end{equation}
where the image group is discussed further below.

Through this action there is an associated bundle which is a
finite-dimensional manifold, the {\em twisted (fibre) cosphere bundle} as the
associated fibre bundle,
$$ 
{\bf S}= {\bf F}\times_{\PGL(\cF^\bullet(Z;V))} S^*Z \stackrel{\hat \pi}{\longrightarrow} X,
$$
where the action of $\PGL(\cF^\bullet(Z;V))$ on $S^*Z$ is via $p\circ \chi.$ 
The bundle ${\bf S}$ has typical fibre $S^*Z.$

We next show that there exist  twisted (fibre) cosphere bundles that are not 
 (fibre) cosphere bundles of fibre bundles. More precisely, 
any smooth fiber bundle of compact manifolds
\begin{equation}
\xymatrix{
Z\ar[r]&Y\ar[d]^\phi\\
& X,
}
\end{equation}
has an associated fibre bundle with typical fibre $S^*Z$,
\begin{equation}
\xymatrix{
S^*Z\ar[r]&{\bf S}\ar[d]^{\hat\phi}\\
& X.
}
\end{equation}

\begin{rems}\label{Contact}
Let $\theta$ denote the contact 1-form on $S^*Z$. Then recall that the contact diffeomorphism group is defined as,
$$
\Con(S^*Z) = \left\{(f, e^h) \in {\rm Diffeo}(S^*Z) \times \CI(S^*Z; \RR^+) \big| f^*\theta = e^h \theta\right\}.
$$
The strict contact diffeomorphism group is defined as,
$$
S\Con(S^*Z) = \left\{f \in {\rm Diffeo}(S^*Z)  \big| f^*\theta =  \theta\right\}.
$$
$S\Con(S^*Z) \hookrightarrow \Con(S^*Z)$ is a subgroup, where $f \mapsto (f,1)$. 
There is a surjective map $\Con(S^*Z) \twoheadrightarrow \CI(S^*Z; \RR^+)$ given by 
$(f, e^h) \mapsto e^h$.
%
%
Therefore we have a principal fibration $S\Con(S^*Z) \hookrightarrow \Con(S^*Z) \twoheadrightarrow \CI(S^*Z; \RR^+), $ where the base
is contractible, and hence it is trivialisable, $\Con(S^*Z) \cong S\Con(S^*Z) \times \CI(S^*Z; \RR^+)$. We conclude that 
$\Con(S^*Z)$ and 
$S\Con(S^*Z)$ are homotopy equivalent.
\end{rems}


{{Let $\Sigma_g$ be a compact Riemann surface of genus $g>1$ with hyperbolic metric of 
constant sectional curvature equal to $-1$, and let $\omega$ denote
the volume form of $\Sigma_g$. By the Gauss-Bonnet theorem, one has
\begin{equation}
\int_{\Sigma_g} \frac{\omega}{2\pi} = 2g-2 \in  \ZZ\setminus\{0\}.
\end{equation}
That is, $\frac{\omega}{2\pi} $ is an integral 2-form representing the 1st Chern class $c_1(S^*Z)$ of the 
cosphere bundle $S^*\Sigma_g$, which is a principal circle bundle over $\Sigma_g$. By \cite{Ban78} and Remarks \ref{Contact},  
the group  $\Con_0(S^*\Sigma_g)$ is homotopy equivalent to $S^1$. By \cite{Earle-Eells69,Gramain73}, 
the group ${\rm Diff}^+(\Sigma_g)$ is homotopy equivalent to the discrete Modular group, ${\rm Mod}_g$. 
It follows that the groups ${\rm Diff}^+(\Sigma_g)$ and  $\Con_0(S^*\Sigma_g)$ and {\em not} homotopy equivalent.
Since $\pi_1(\Con(S^*\Sigma_g)) =\pi_1(\Con_0(S^*\Sigma_g))
 \cong \ZZ$, there are infinitely many topologically distinct principal $\Con(S^*\Sigma_g)$ bundles over $S^2$, 
 and therefore  infinitely many topologically distinct twisted (fibre) cosphere bundles with typical fibre $S^*\Sigma_g$,
 but
 there are {\em no} non-trivial fibre bundles over $S^2$ with typical fibre $\Sigma_g$. 
 }}


 \section{Chern class of the Fourier integral operator central extension}
\label{CentralextFIO}
 

Using a regularized trace on pseudodifferential operators, we describe a
natural class of connections on the central extension in equation
\eqref{centralext}.  The curvature is expressed in terms of the residue
trace $\rTr$ of Guillemin and Wodzicki and the exterior derivation. We compute the
1-form arising from the  central extension which together with the curvature 
determines the central extension up to isomorphism in
the absence of torsion in degree 2 integral cohomology, see
\cite{BM,MS}.
 
Consider the Maurer-Cartan form, which is a
Lie algebra valued differential 1-form
$$
\Theta \in\Omega^1(\GL(\cF^0(Z;V));\LA(Z;V)),
$$
where $\LA(Z;V)$ is the Lie algebra of $\GL(\cF^0(Z;V)),$ consisting of the
pseudodifferential operators of order $1$ with the additional constraint
that the principal symbol is scalar 
and purely
imaginary. This canonical form is determined by left-invariance and the
identification of $T_{\Id}\GL(\cF^0(Z;V))$ with $\LA(Z;V):$
\begin{equation}
L_g^*\Theta=\Theta,\ \Theta _{\Id}(v)=v.
\label{MC1}\end{equation}
Since $\cF^0(Z;V)$ lies in a linear space of operators, it is directly
meaningful to use the familiar notation $\Theta = F^{-1}dF$ for $F\in \GL(\cF^0(Z;V)).$
Under the right action and the adjoint action on the Lie algebra
\begin{equation}
R _g^*(\Theta) = {\rm Ad}(g)\Theta=g\Theta g^{-1}.
\label{invariance}\end{equation}

Let $Q\in\Psi^1(Z;V)$ be an elliptic pseudodifferential operator which is
self-adjoint and positive with respect to some inner product and density
and let
$$
\Tr_Q : \Psi^\bbZ(Z;V) \longrightarrow \bbC
$$
denote the {\em regularized trace} with respect to $Q.$ That is, $\Tr_Q(A)$
is the regularized value at $z=0$ of the meromorphic extension of $\Tr(Q^zA).$ 
Thus $\Tr_Q$ extends the operator trace from the ideal of trace class
operators $\Psi^{m}(Z;V),$ $m<-\dim(Z),$ but is not itself a trace.

As shown by Guillemin \cite{G93} one can use in place of $Q^z$ any entire
classical family $Q(z)\in\Psi^z(Z;W)$ with $Q(0)=\Id$ and entire inverse
$Q(z)^{-1}\in\Psi^{-z}(Z;W).$ Then the residue trace of Guillemin and
Wodzicki and the corresponding regularized trace are determined by the
expansion near $z=0$
\begin{equation}
\Tr(Q(z)A)=z^{-1}\rTr(A)+\Tr_{Q}(A)+zG(z).
\end{equation}

The evaluation of the residue trace on a commutator is given by the trace-defect formula 
\begin{equation}
\begin{gathered}
\Tr_Q([A,B])=\rTr([\log Q,A]B)=\rTr((\delta _QA)B)=-\rTr(A\delta _QB)\Mwhere\\
\delta _QA=[\log Q,A]=\frac{d}{dz}\big|_{z=0}Q(z)AQ(z)^{-1}
\end{gathered}
\label{20.August.2012.1}\end{equation}
is the exterior derivation defined by $Q(z).$

Under change of the regularizing operator $Q,$ to $Q',$ the regularized trace
changes by 
\begin{equation}
\Tr_{Q'}(A)-\Tr_{Q}(A)=\rTr(A(\log Q'-\log Q)),\ \log Q'-\log Q\in\Psi^0(Z;V).
\label{20.August.2012.2}
\end{equation}

\begin{lemma}\label{20.August.2012.3} There is an entire holomorphic family
$Q(z)\in\Psi^z(Z;V)$ with entire inverse $Q(-z)$ such that
\begin{equation}
\Tr_Q(\Id)=1.
\label{20.August.2012.4}
\end{equation}
\end{lemma}

\begin{proof} Since $(Q^z)^*=Q^{\overline{z}}$ the regularized trace of the
  identity, which is to say the value at $z=0$ of the zeta function, is
  real. If $A\in\Psi^{-\dim Z}(Z;V)$ is self-adjoint with scalar symbol
  then there is a self-adjoint smoothing operator $R\in\Psi^{-\infty}(Z;V)$
  such that $\Id+\ha A+R$ is positive. The regularizing family $Q'(z)=(\Id+\ha
  A+R)^zQ^z(\Id+\ha A+R)^z\in\Psi^z(Z;V)$ corresponds to the exterior derivation
\begin{equation}
\log Q'=\log Q+A+2R
\label{20.August.2012.5}
\end{equation}
and so shifts the regularized trace of the identity by $\rTr(A).$ Since
this is given by the integral of the trace of the leading symbol of $A$ the
normalization \eqref{20.August.2012.4} can be arranged.
\end{proof}

As a functional on $\LA(Z;V)\subset\Psi^{\bbZ}(Z;V),$ the regularized
trace acts on the range of the Maurer-Cartan form, so
\begin{equation}
A_Q=\Tr_Q(\Theta)\in\Omega^1(\GL(\cF^0(Z;V)))
\label{20.August.2012.6}\end{equation}
is a well-defined smooth 1-form on $\GL(\cF^0(Z;V)).$

\begin{lemma}\label{lemma:conn-central} If the regularized trace is chosen
  to satisfy \eqref{20.August.2012.4} then the 1-form $A_Q$ in
  \eqref{20.August.2012.6} is a connection 1-form on the  central extension
  in equation \eqref{centralext} and if $Q_1$ is another such normalized
  regularizing family then
\begin{equation}
A_Q - A_{Q_1} = - \rTr(\Theta(\log(Q)-\log(Q_1)))\in\Omega^1(\PGL(\cF^0(Z;V))).
\label{eqn:varregtr}\end{equation}
\end{lemma}

\begin{proof} Clearly $A_Q$ is a left-invariant 1-form on $\GL(\cF(Z;V)).$
  The central subgroup consists of the multiples of the identity, so the
  normalization \eqref{20.August.2012.6} ensures that it restricts to the
  fibres of $\GL(\cF^0(Z;V))$ over $\PGL(\cF^0(Z;V))$ to be the
  Maurer-Cartan form for the centre i.e.\ it is a connection form. The
  transgression formula \eqref{eqn:varregtr} follows and since $\Theta$ is
  a multiple of the identity, the difference vanishes on the fibres and
  hence is a smooth form on $\PGL(\cF^0(Z;V)).$ 
\end{proof}

\begin{lemma}\label{lemma:curv-central} For a normalized regularization,
  satisfying \eqref{20.August.2012.4}, the curvature of the connection
  $A_Q$ is
\begin{equation}
\Omega_Q= \Tr_R(\delta_Q(\Theta)\wedge\Theta)
\label{20.August.2012.7}\end{equation}
and transgresses under change of regularization by
$$
\Omega_Q-\Omega_{Q_1} = - d\Tr_R\left(\Theta(\log(Q) - \log(Q_1))\right).
$$
The first Chern class of the  central extension is therefore  
\begin{equation}
c_1(\GL(\cF(Z;V))) = \left[\frac{i}{2\pi} \Omega_Q\right] \in \bfH^2(\PGL(\cF(Z;V); \ZZ).
\end{equation}

\end{lemma}

\begin{proof} It suffices to compute the curvature as a form on $\GL(\cF^0(Z;V)).$ 
For $\psi_1,$ $\psi_2 \in\LA(Z;V)$ the standard formula for the
differential gives
$$
d\Tr_Q(\Theta)(\psi_1, \psi_2)=\psi_1\Tr_Q(\Theta)(\psi_2) - \psi_2  \Tr_Q(\Theta)(\psi_1)    -\Tr_Q(\Theta)([\psi_1, \psi_2]).
$$
Since the Maurer-Cartan 1-form $\Theta$ is left-invariant the first two
terms on the right side vanish. Applying the trace-defect formula
\begin{equation}
\Tr_Q(\Theta)([\psi_1,\psi_2]) 
=\Tr_Q([\psi_1,\psi_2])=\rTr(\delta_Q(\psi_1)\psi_2)
\end{equation}
gives the formula \eqref{20.August.2012.7} in general by left-invariance.

The transgression formula follows immediately from Lemma~\ref{lemma:conn-central} above.
\end{proof}

Note that the curvature can be expanded to more resemble a suspended Chern
form (cf. \cite{MelRoc}) in which $\delta _Q$ plays the role of the push-forward. Namely
\begin{equation}
\Omega_Q(F)=-\Tr_R(F^{-1}(\delta_QF)F^{-1}dF\wedge
F^{-1}dF)+d\Tr_R(F^{-1}(\delta_QF)F^{-1}dF).
\label{27.August.2012.1}\end{equation}

The circle bundle given by a central extension of a group does not in
general characterise the central extension, cf.\ \cite{BM, MS}. The
multiplicative structure, i.e.\ the primitivity of the bundle, gives as a
further invariant, namely a 1-form on the product of the base group with itself.

For a Lie group (possibly infinite-dimensional as in this case) $G$ and a form
$\beta$ on $G^p$ consider the form on the product
$G^{p+1}$ defined by
\begin{equation}
\begin{gathered}
\delta\beta=\sum\limits_{j=1}^{p+1}(-1)^{j-1}d_i^*\beta,\ d_i:G^{p+1}\longrightarrow G^p\\
d_1(g_1,\dots,g_{p+1})=(g_2,\dots,g_{p+1}),\\
d_i(g_1,\dots,g_{p+1})=(g_1,\dots,g_{i-1}g_i,\dots,g_{p+1}),\ 1<i\le p\\
d_{p+1}(g_1,\dots,g_{p+1})=(g_1,\dots,g_{p}).
\end{gathered}
\label{27.August.2012.3}\end{equation}
Then, 
\begin{equation}
d\delta =\delta d,\ \delta ^2=0.
\label{27.August.2012.2}\end{equation}

\begin{lemma}\label{lemma:1-form} The 1-form
\begin{equation}
\alpha_Q(F_1,F_2)= - \Tr_R\left((\delta _QF_2)F_2^{-1}F_1^{-1}dF_1)\right)
\label{27.August.2012.4}\end{equation}
is well-defined on $(\PGL(\cF^0(Z;V)))^2$ and satisfies
$$
d\alpha_Q = \delta \Omega_Q,\ \delta\alpha_Q = 0.
$$
\end{lemma}

\noindent In the absence of torsion such a pair $(\Omega_Q, \alpha_Q)$ 
determines the central extension up to isomorphism, see \cite{BM, MS}.

\begin{proof} The operator in \eqref{27.August.2012.3} is well-defined on
  any smooth group, let $\hat\delta$ denote the corresponding operator on
  the full group $\GL(\cF^0(Z;V)).$ Let $\Theta_i,$ for $i=1,2$ be the
  pull-backs to $(\GL(\cF^0(Z;V))^2$ of the Maurer-Cartan form from the two factors. Then
\begin{equation}
\begin{aligned}
(\hat\delta A_Q)(F_1,F_2)=
&\Tr_Q(\Theta_2)-\Tr_Q(F_2^{-1}F_1^{-1}d(F_1F_2))+\Tr_Q(\Theta_1)\\
=&-\Tr_Q([F_2^{-1},\Theta_1F_2])\\
=&-\rTr((\delta _QF_2)F_2^{-1}F_1^{-1}dF_1).
\end{aligned}
\label{22.August.2012.5}\end{equation}
This is just $\alpha_Q$ in \eqref{27.August.2012.4}.

Restricting to a fibre of \eqref{centralext} in the first factor
corresponds to fixing $F_1$ up to a scalar multiple, so $\Theta_1$ is a
scalar (1-form) multiple of the identity and $\alpha _Q$ is therefore a
multiple of $\rTr((\delta_Q F)F^{-1})$ pulled-back from the second
factor. By the trace-defect formula this is equal to $\Tr_Q([F^{-1},F])$
and so vanishes. Similarly, restricted to the fibre in the first factor
$(\delta _QF)F^{-1}$ is constant, so in fact $\alpha _Q$ descends to a form
on $(\PGL(\cF^0(Z;V)))^2$ as claimed. In view of \eqref{27.August.2012.2}
this form satisfies the desired identities.
\end{proof}


\section{The Dixmier-Douady class of a projective FIO bundle}\label{sect:FIODD}


Next we proceed to compute the image in deRham cohomology of the
Dixmier-Douady class of a principal $\PGL(\cF^0(Z;V))$-bundle over
$X.$ This is the obstruction to lifting the structure group to
$\GL(\cF^0(Z;V))$ and to compute it we use the approach of Murray, via the
associated bundle gerbe, as elaborated in \cite{MS} for the case of loop groups.

The principal bundle, ${\bf F},$ can be seen as a `lifting gerbe' with the
central extension of its structure group written as a principal $\bbC^*$-bundle
\begin{equation}
\xymatrix{
\CC^*\ar[r]&
\GL(\cF^0(Z;V))\ar[d]^{}\\
&\PGL(\cF^0(Z;V))\ar@{-}[r]&
{\bf F}\ar[d]^{\pi}\\
&& X.}
\label{lifting}\end{equation}
A connection on the $\bbC^*$ bundle over the group was constructed in
Section~\ref{CentralextFIO} with curvature given by
\eqref{20.August.2012.7}. The associated bundle gerbe is the fibre product
${\bf F}^{[2]}$ with line bundle $\cL=\tau^*\GL(\cF^\bullet(Z;V))$ obtained
by pulling back via the fibre shift map $\tau:{\bf F}^{[2]}\longrightarrow
\PGL(\cF^0(Z;V)),$ $\tau(a, b) = b^{-1}a;$ there are two `simplicial' maps
back to the principal bundle
\begin{equation}\label{bg}
\xymatrix{
&
\cL\ar[d] \ar[r]
&
\GL(\cF^\bullet(Z;V))\ar[d]^{}
\\
{\bf F}\ar[d]_{\pi}
&
{\bf F}^{[2]}\ar[dl]^{\pi^{[2]}}
\ar@<3pt>[l]^-{\pi_2}\ar@<-3pt>[l]_-{\pi_1}
\ar[r]^-\tau
&
\PGL(\cF^\bullet(Z;V))
\\
X. &
}
\end{equation}

We proceed to construct a connection on $\cL$ as a principal $\bbC^*$
bundle over $\bfF^{[2]}$ which is \emph{primitive} in the sense that its
curvature decomposes in the form $\pi_1^*\bfB-\pi_2^*\bfB$ where $\bfB$ is a 2-form
on $\bfF,$ the curving or B-field. Then, $d\bfB$ is necessarily basic, hence
is the pull-back of a 3-form, $H,$ on $X$ which represents the image of the
Dixmier-Douady class in $\bfH^3(X;\bbR).$ 

To get an explicit formula for $H$ we start by modifying the connection
$A_Q$ of Section~\ref{CentralextFIO} to the `middle' connection form on
$\GL(\cF^0(Z;V))$
\begin{equation}
\alpha=\ha\Tr_Q(\hat\theta _L+\hat\theta_R),\
\hat\theta_L=g^{-1}dg,\ \hat\theta _R=(dg)g^{-1}=g\theta _L g^{-1} 
\label{27.August.2012.19}\end{equation}
corresponding to the choice of a normalized regularized trace $\Tr_Q$ on
$\Psi^{\bbZ}(Z;V)$ and hence on the Lie algebra of $\GL(\cF^0(Z;V)).$

The derivation $\delta _Q$ associated to the holomorphic family $Q$ induces
a map from $\GL(\cF^0(Z;V))$ to its Lie algebra by
\begin{equation}
\zeta=g^{-1}\delta _Qg.
\label{27.August.2012.20}\end{equation}
Since $\delta _Q\Id=0,$ this function descends to $\PGL(\cF^0(Z;V))$ and
then takes values in its Lie algebra: 
\begin{equation}
\zeta:\PGL(\cF^0(Z;V))\longrightarrow\PLA=\LA(\PGL(\cF^0(Z;V)).
\label{27.August.2012.21}\end{equation}

As in the relation of a connection form on $\bfF$ and the Maurer-Cartan
form, we consider a Higgs field related to $\zeta.$ That is, a {\em Higgs field} is a smooth map 
\begin{equation}
\bfPhi:\bfF\longrightarrow\PLA
\label{27.August.2012.22}\end{equation}
with the transformation law  
\begin{equation}
g^*\bfPhi=\zeta(g)+g^{-1}\bfPhi g,\qquad g\in\PGL(\cF^0(Z;V)).
\label{27.August.2012.23}\end{equation}
In terms of a local trivialization of $\bfF,$ over $U\subset X,$ as a
principal bundle this is equivalent to requiring 
\begin{equation}
\bfPhi=\zeta+g^{-1}\phi(x)g
\label{27.August.2012.24}\end{equation}
for a local field $\phi,$ on $U,$ with values in $\PLA.$ Since this condition is
preserved under convex combinations such a field $\bfPhi$ can be constructed,
as for a connection, as the sum of the local fields $\zeta$ over a
partition of unity subordinate to trivializations. The difference between
two Higgs fields associated to $\delta_Q$ in this sense is a section of the
adjoint bundle associated to $\bfF$ over $X.$

There is also a somewhat more natural construction of such a Higgs field which we
briefly indicate. The bundle $\bfPsi^{\bbZ}$ of pseudodifferential
operators associated to $\bfF$ (by construction) can be extended to a
bundle of complex-order classical operators, $\bfPsi^{\bbC},$ since the
action by conjugation of $\PGL(\cF^0(Z;V))$ extends from $\Psi^{\bbZ}(Z;V)$
to $\Psi^{\bbC}(Z;V).$ Choices of normalized entire regularizing family
$Q(z)\in\Psi^{z}(Z;V)$ as in Lemma~\ref{20.August.2012.3} for a covering by
trivializations of $\bfF$ can be patched through a partition of unity to
give a section $\tilde Q(z)\in\CI(X;\bfPsi^{z}(Z;V))$ with the desired
properties on each fibre, including the normalization condition
\eqref{20.August.2012.4}. Then  
\begin{equation}
\Phi(F)=\frac{d}{dz}\big|_{z=0}F^{-1}\tilde Q(z)F
Q(z)^{-1},\ \Phi:\bfF\longrightarrow \PLA(\cF^0(Z;V))
\label{27.August.2012.32}\end{equation}
is a Higgs field for the derivation associated to the fixed choice $Q(z)$
of regularizing family in $\Psi^{\bbC}(Z;V).$ 

Denote the pull back of $\zeta$ to $\bfF$ under $\tau$ as $\Zeta;$ a choice of
Higgs field associated to $\delta _Q$ provides a global splitting in the sense that
\begin{equation}
\Zeta=\bfPhi_1-\tau^{-1}\bfPhi_2\tau,\ \bfPhi_i=\pi_i^*\bfPhi
\label{27.August.2012.25}\end{equation}
since this follows from \eqref{27.August.2012.24} in any local
trivialization. There is also a relation between $\Zeta$ and the pull-back of
the Maurer-Cartan form
\begin{equation}
d\Zeta=\delta_Q\Theta+[\Zeta,\Theta].
\label{27.August.2012.18}\end{equation}
This follows from the corresponding formula on $\PGL(\cF^0(Z;V))$ that 
\begin{equation}
d\zeta =d(g^{-1}\delta _Qg)=\delta _Q(g^{-1}dg)+[\zeta,g^{-1}dg]=\delta
_Q(\theta)+[\zeta,\theta].
\label{27.August.2012.30}\end{equation}

To capture the contribution of the geometry choose a smooth connection form
${\bf A}$ on ${\bf F}$ as a principal $\PGL(\cF^0(Z;V))$ bundle. Then the
pull-back of the Maurer-Cartan form can be expressed in terms of the two
pull-backs of the connection form
\begin{equation}
\Theta=\bfA_1-\tau ^{-1}\bfA_2\tau,\qquad \bfA_i=\pi_i^*\bfA.
\label{27.August.2012.27}\end{equation}
Indeed this follows from the expression for a connection in terms of a
local trivialization
\begin{equation}
{\bf A}=g^{-1}dg+g^{-1}\gamma (x)g.
\label{27.August.2012.26}\end{equation}

The `symmetric' connection form in \eqref{27.August.2012.19} pulls back to
a connection on the $\bbC^*$ bundle $\cL.$ The differential of $\alpha$ as
a form on $\PGL$ follows from the trace-defect formula
\begin{equation}
d\alpha =\frac12\Tr_Q(-\hat\theta _L\wedge \hat\theta  _L+\hat\theta _R\wedge \hat\theta
_R)=\pi^*w,\ w=\pi^*\frac12\rTr(\zeta \theta\wedge \theta).
\label{10.9.2012.2}\end{equation}
Here $\theta_L$ has been replaced by the (left) Maurer-Cartan form
$\theta$ on $\PGL$ which is justified since any vertical part in $\hat\theta
_L\wedge\hat\theta_L$ takes values as a multiple of the identity in $\LA(\cF^0(Z;V))$
and so vanishes with $\rTr(\zeta);$ the resulting basic form is fibre constant.

The pull-back of $d\alpha$ to $\bfF^{[2]}$ may therefore be written 
\begin{equation}
W=\tau^*w =\ha\rTr(\Zeta\Theta\wedge \Theta).
\label{10.9.2012.3}\end{equation}

Thus the differential of the connection form $\tau^*\alpha$ on the pull-back
$\bbC^*$-bundle $\cL$ is the pull-back of $\pi^*W$ from $\bfF^{[2]}$ where
$W=\tau^*w$ is the pull-back of the $d\alpha$ from $\PGL(\cF^0(Z;V)).$ 
This is not a primitive connection; however

\begin{proposition}\label{10.9.2012.6} The connection form
  $\tau^*\alpha-\pi^*\mu,$ where 
\begin{equation*}
\mu=\ha d\rTr(\Zeta(\bfA_1+\tau ^{-1}\bfA_2\tau),
\label{27.August.2012.31}\end{equation*}
is primitive on $\cL$ with differential the pull-back of $\bfB_1-\bfB_2,$ where
$\bfB_i=\pi_i^*\bfB,$ for the B-field 
\begin{equation}
\bfB=-\ha\rTr((\delta _Q\bfA)\wedge \bfA)+\rTr(\bfPhi\bfW),
\label{10.9.2012.7}\end{equation}
with $\bfPhi$ a Higgs field on $\bfF$ for the derivation $\delta_Q$ and
$\bfW=d\bfA+\bfA\wedge\bfA$ the curvature form for the connection $\bfA$ on $\bfF.$
\end{proposition}

\begin{proof} We wish to show that 
\begin{equation}
W-d\mu=\bfB_1-\bfB_2
\label{10.9.2012.8}\end{equation}
so we start by expanding $d\mu:$ 
\begin{multline}
2d\mu=
\rTr(d\Zeta\wedge(\bfA_1+\tau^{-1}\bfA_2\tau))
+\rTr(\Zeta\wedge(\bfW_1+\tau^{-1}\bfW_2\tau)\\
-\rTr(\Zeta(\bfA_1\wedge \bfA_1+\tau ^{-1}\bfA_2\wedge\bfA_2\tau))
-\rTr(\Zeta(\Theta\tau ^{-1} \bfA_2\tau+\tau ^{-1}\bfA_2\tau \Theta)
\label{10.9.2012.9}\end{multline}
since $d\tau=\tau\Theta.$ Now, using \eqref{27.August.2012.18} to expand $d\Zeta,$ 
the first term on the right in \eqref{10.9.2012.9},
with the wedge product now understood, can be written as,
\begin{equation}
\begin{aligned}
\rTr(d\Zeta(\bfA_1+\tau^{-1}&\bfA_2\tau))\\
=&\rTr(\delta_Q\Theta(\bfA_1+\bfA_2'))
+\rTr([\Zeta,\Theta]((\bfA_1+\bfA_2'))\\
=&-\rTr(\Theta(\delta_Q\bfA_1+\tau^{-1}(\delta _Q\bfA_2)\tau))
+\rTr(\Theta([\Zeta,\bfA_2']+[\Zeta,\Theta](\bfA_1+\bfA_2'))\\
=&\rTr((\delta_Q\bfA_1)\bfA_1)
-\rTr((\delta _Q\bfA_2)\bfA_2)
+2\rTr(\Zeta\bfA_1\bfA_1)
\end{aligned}
\label{10.9.2012.11}\end{equation}
where $\bfA_2'=\tau^{-1}\bfA_2\tau.$
Here the identity $\rTr\circ \delta_Q\equiv0$ has been used in the form
\begin{equation*}
\begin{gathered}
\rTr((\delta_Q\bfA_1)\bfA_2'+\bfA_1t^{-1}(\delta _Q\bfA_2)\tau)=
\rTr(\Zeta\bfA_1\bfA_2')+\rTr(\Zeta\bfA_2'\bfA_1)\\
\rTr((\delta _Q\Theta)(\bfA_1+\bfA_2'))=
\rTr((\delta _Q(\bfA_1+\bfA_2')\Theta).
\end{gathered}
\end{equation*}
The last term on the right in \eqref{10.9.2012.9} expands to 
\begin{equation*}
2\rTr(\Zeta\tau ^{-1}\bfA_2\bfA_2\tau)
-\rTr(\Zeta\bfA_1\tau ^{-1}\bfA_2\tau)
-\rTr(\Zeta\tau ^{-1}\bfA_2\tau\bfA_1)
\label{10.9.2012.13}\end{equation*}
therefore
\begin{multline}
2d\mu=\rTr((\delta_Q\bfA_1)\bfA_1)
-\rTr((\delta _Q\bfA_2)\bfA_2)
+\rTr(\Zeta(\bfW_1+\tau^{-1}\bfW_2\tau)+\rTr(\Zeta\Theta \Theta).
\label{10.9.2012.14}\end{multline}

The curvature form $\bfW$ of $\bfA$ is a 2-form with values in the adjoint
bundle so $\tau^{-1}\bfW_2\tau=\bfW_1$ and computing in a local
trivialization where $\bfW=g^{-1}w(x)g$ for a $2$-form $w$ on the base
(with values in the Lie algebra)
$$
\begin{aligned}
\rTr(\Zeta\bfW_1)
&=\rTr(\zeta _1a^{-1}w(x)a-\tau^{-1}\zeta_2\tau(a^{-1}w(x)a)\\
&=\rTr(\zeta _1a^{-1}w(x)a-\zeta _2b^{-1}w(x)b)\\
&=\rTr(\bfPhi_1\bfW_1-\bfPhi_2\bfW_2)
\label{10.9.2012.15}\end{aligned}
$$
for any Higgs field for the derivation $\delta _Q,$
$\bfPhi=z+p^{-1}\phi(x)p$ in the trivialization. Finally then
\eqref{10.9.2012.14} can be rewritten in the form \eqref{10.9.2012.8} 
%
$$
W-d\mu=\left(-\ha\rTr((\delta_Q\bfA_1)\bfA_1)+\rTr(\bfPhi_1\bfW_1)\right)
-\left(-\ha\rTr((\delta _Q\bfA_2)\bfA_2)+\rTr(\bfPhi_2\bfW_2)\right)
$$
%
as claimed.
\end{proof}

With $Q$ fixed, a change of Higgs field to $\bfPhi+\psi$ where $\psi$ is a
section of the adjoint bundle, changes $\bfB$ in \eqref{10.9.2012.7} by a
2-form on $X:$
\begin{equation*}
\bfB+\rTr(\psi\bfW),\quad \rTr(\psi\bfW)\in\CI(X;\Lambda^2).
\label{27.August.2012.28}\end{equation*}
Changing $\bfA$ to another connection $\bfA+\lambda,$ where $\lambda$ is a
section of the adjoint bundle with values in the pull-back of the $1$-form
bundle on $X$ changes the curvature form to
$\bfW+d\lambda+[\bfA,\lambda]+\lambda \wedge \lambda $ and hence changes the
B-field to 
\begin{equation}
\bfB-\ha\rTr((\delta _Q\lambda)\lambda) +
\rTr(\bfPhi(d\lambda+[\bfA,\lambda]+\lambda \wedge \lambda))
\label{27.August.2012.29}\end{equation}
Under change of the normalized derivation on the Lie algebra, the B-field changes to 
\begin{equation}
\bfB-\frac{i}{2\pi} \Tr_R(\tfrac{1}{2} \bfA\wedge [P, \bfA] )
\end{equation}
where $P = \log(Q)-\log(Q')\in\Psi^0(Z;V).$

The bundle gerbe with primitive connection has the property that $d\bfB,$ with
$\bfB$ given here by \eqref{10.9.2012.7}, is necessarily the pull-back of a
3-form, $H$ on $X,$ which represents the deRham class of the Dixmier-Douady
invariant.
  
\begin{theorem}\label{27.August.2012.11} If $\bfF$ is a principal
$\PGL(\cF^0(Z;V))$-bundle over $X,$ the image of the Dixmier-Douady class, for the
central extension \eqref{centralext}, in $\bfH^3(X;\RR)$ is represented by the 3-form  
\begin{equation}
H=H_{Q, \bfPhi,\bfA}=-\frac{i}{2\pi}\rTr(\bfW\wedge\nabla^Q \bfPhi),\qquad
\nabla^Q \bfPhi = d\bfPhi + [\bfA, \bfPhi] - \delta_Q\bfA
\label{27.August.2012.12}\end{equation}
where $\rTr$ is the residue trace on the Lie algebra, $\bfA$ is a connection
on $\bfF$ with curvature $\bfW$ and $\bfPhi$ is a Higgs field on $\bfF$
for the normalized derivation $\delta_Q$ on $\PLA(\cF^0(Z;V)).$

In fact, the triple $(H, \bfB, \tau^*\alpha-\pi^*\mu)$ determines 
the Dixmier-Douady class of $\bfF$, for the
central extension \eqref{centralext},  in $\bfH^3(X;\ZZ)$.
\end{theorem}

\begin{proof} As noted above it suffices to compute the deRham differential
  of the B-field in \eqref{10.9.2012.7}:
$$
d\bfB =\frac{i}{2\pi}\left(\ha\rTr(d\bfA\wedge\delta_Q\bfA)
-\ha\rTr(\bfA\wedge\delta_Q d\bfA)+\rTr(\bfPhi\wedge d\bfF +d\bfPhi\wedge \bfF)\right).
$$
Using the trace property and Bianchi identity $d\bfW=[\bfW,\bfA]$
$$
d\bfB =\frac{i}{2\pi}\Tr_R(d\bfA\wedge\delta_Q\bfA-
\bfW\wedge[\bfA,\bfPhi]-\bfW\wedge d\bfPhi),
$$
and since the residue trace of $\bfA\wedge\bfA\wedge\delta_Q\bfA$ vanishes,
$$
d\bfB=\frac{i}{2\pi}\Tr_R(\bfW\wedge\delta_Q\bfA-
\bfW\wedge[\bfA,\bfPhi]-\bfW\wedge d\bfPhi).
$$
This descends to a closed 3-form on $X$ and so gives \eqref{27.August.2012.12}.

For the last statement of the Theorem, observe that the triple $(H, \bfB, \tau^*\alpha-\pi^*\mu)$ determines 
the Deligne class (cf. \cite{Brylinski}) of $\bfF$, which in particular determines its Dixmier-Douady class in $\bfH^3(X;\ZZ)$. 
\end{proof}

The transgression formula for $H_Q$ under change of $Q$ is given by
\begin{equation}\label{eqn:3-curvature}
H_{Q, \bfPhi, \bfA} - H_{Q', \bfPhi, \bfA} = - \frac{i}{2\pi}d\Tr_R(\bfA\wedge[P,\bfA]),
\end{equation}
where $P = \log(Q)-\log(Q').$

The other transgression formulae are:-
\begin{equation}
H_{Q, \bfPhi, \bfA'} - H_{Q, \bfPhi,\bfA} = -\frac{i}{2\pi}d \Tr_R(f \wedge \nabla^Q\bfPhi)
\end{equation}
where $\bfA'-\bfA=f\in \Omega^1(X,\End({\bfPsi^\bbZ})).$ If $\bfPhi'$ is
another choice of Higgs field and $\bfPhi'- \bfPhi = \sigma$, where
$\sigma$ is  a section of the adjoint bundle, then
\begin{equation}
H_{Q, \bfPhi',\bfA} - H_{Q, \bfPhi, \bfA} =-\frac{i}{2\pi} d\Tr_R(\bfW\wedge\sigma).
\end{equation}

Note that $H$ is also the Dixmier-Douady class of the bundle of compact
operators obtained by closing $\bfPsi^{-\infty}$ in the topology of bounded
operators on $L^2(Z;V).$

There are higher characteristic classes for the pseudodifferential algebra
bundle $\bfPsi^\bbZ$ over $X$, which are represented by $H_Q(n)=c_n
\rTr(\bfW^n\wedge\nabla^Q\bfPhi)$ for appropriate constants $c_n.$ These
are closed differential forms of degree$2n +1$ on $X.$ Explicit
transgression formulae can be derived for these forms as above. The
characteristic classes represented by these differential forms may be
non-trivial in cohomology, since $\GL(\cF^0(Z;V))$ need not be
contractible. Analogous results for loop groups have been obtained in
\cite{Vozzo}.

\section{Examples}\label{sect:examples}

The examples below illustrate that the rationalized Dixmier-Douady
invariant computed in Theorem \ref{27.August.2012.11} above does not in general vanish, i.e.\ the
Dixmier-Douady invariant itself may be non-torsion in this setting of
pseudodifferential bundles.

\subsection{Finite dimensional bundle gerbes}
Here we will define the smooth Azumaya bundle and the filtered algebra bundle of 
pseudodifferential operators in the case of a  finite dimensional bundle gerbe, thereby 
giving a large class of examples that satisfy the hypotheses of the main theorem in the paper.

The data we use to define a smooth Azumaya bundle is:-

\begin{itemize}
\item A smooth fiber bundle of compact manifolds
\begin{equation}
\xymatrix{
Z\ar[r]&Y\ar[d]^\phi\\
& X.
}
\end{equation}

\item A primitive line bundle  $J$ over $Y^{[2]}$, in the sense that under
lifting by the three projection maps $\pi_j: Y^{[3]} \to Y^{[2]}$, which omits the 
$j$-th factor,
\begin{equation}
\xymatrix{Y^{[3]}\ar@<2ex>[r]^{\pi_1}
\ar[r]^{\pi_2}
\ar@<-2ex>[r]^{\pi_3}& Y^{[2]}}
\end{equation}
%
there is an isomorphism of line bundles on $Y^{[3]}$
\begin{equation}
\pi_3^*J\otimes\pi_1^*J\cong \pi_2^*J.
\label{gerbeprod}\end{equation}
\item Equation \eqref{gerbeprod} can be viewed as a product, and this product is required to be associative, which is a compatibility condition on $Y^{[4]}$,
\begin{equation}
\xymatrix{
 J_{(y_1,y_2)} \otimes J_{(y_2,y_3)}  \otimes J_{(y_3,y_4)}\ar@{-}[r]\ar[d]^-{\sigma}& J_{(y_1,y_3)}  \otimes J_{(y_3,y_4)} \ar[d]^-{\sigma}\\
J_{(y_1,y_2)}  \otimes J_{(y_2,y_4)} \ar@{-}[r]&J_{(y_1,y_4)},
}
\end{equation}
for all $(y_1, y_2, y_3, y_4) \in Y^{[4]}$.
\end{itemize}
 The data above specifies a {\em bundle gerbe} $(Y/X, J)$ (cf. \cite{MS2}), which in turn 
  determines an infinite rank {\em smooth Azumaya bundle}, $\smA \to X,$  defined in terms of its space of global
sections
\begin{equation}
\CI(X;\smA)=\CI(Y^{[2]};J).
\end{equation}
A local description of this identification is given, in particular, in the proof of Proposition \ref{prop:ex}  below.
It has fibres isomorphic to the algebra of
smoothing operators on the fibre, $Z,$ of $Y$ with Schwartz kernels
consisting of the smooth sections of the primitive line bundle $J$ over $Z^2.$
The primitivity property of $J$ ensures that $\CI(X;\smA)$ is an algebra. 
The completion of this algebra of `smoothing operators' to a bundle with
fibres modelled on the compact operators has Dixmier-Douady invariant
$\delta(Y/X;J) \in \bfH^3(Y;\bbZ).$ 

We now  define the associated projective bundle of
pseudodifferential operators. We do this by direct generalization of
the definition of the smooth Azumaya bundle $\smA$ above.
For any $\bbZ_2$-graded bundle
$\bbE=(E_+,E_-)$ over $Y$ set the projective filtered algebra bundle of 
pseudodifferential operators
\begin{equation}
\Psi^{\ZZ}_J(Y/X; \hom(\bbE))=
I^{\ZZ}(Y^{[2]},\Diag;\Hom(\bbE)\otimes\Omega_R\otimes J)\longrightarrow X
\end{equation}
where $\Hom(\bbE)=E_-\boxtimes E_+'$ over $Y^{[2]}$ and $I^{\bullet}$ is the
space of (classical) conormal distributions. As is typical in projective index theory,
the Schwartz kernel of the projective family of elliptic operators is globally
defined, even though one only has local families of elliptic operators 
with a compatibility condition on triple overlaps given by a phase factor.  

\begin{proposition}\label{prop:ex} In the situation described above, we have the following equality of 
Dixmier-Douady classes,
$$
\delta(\Psi^{\ZZ}_J(Y/X; \hom(\bbE))) = \delta(Y/X; J) \in \bfH^3(X, \ZZ).
$$
\end{proposition}

\begin{proof}

Now the gerbe associated to the bundle gerbe $(Y/X; J)$ is defined as follows: let $s_i:U_i \to \phi^{-1}({U_i})$ be a local smooth section of the fibre bundle
$Y\stackrel{\phi}{\longrightarrow} X$, where $U_i $ is an open subset of $X$. Then the gerbe associated to the bundle gerbe $(Y/X, J)$ is 
the collection of line bundles $J_{ij} = (s_i, s_j)^*J$ on the double overlaps $U_i\cap U_j$. 

The section $s_i :U_i\longrightarrow \phi^{-1}(U_i),$ induces an isomorphism of $J$ over the open
  subset $V_i = \phi^{-1}(U_i) \times_{U_i} \phi^{-1}(U_i)$ of $Y^{[2]},$ with
\begin{equation}
J\big|_{V_i}\cong_s \Hom(K_{i})=K_{i}\boxtimes K_{i}'
\end{equation}
for a line bundle $K_i$ over $\phi^{-1}(U_i)\subset Y,$ where $K_i'$
denotes the line bundle dual to $K_i.$ Another choice of section
$s_j:U_j\longrightarrow \phi^{-1}(U_j),$ determines another line bundle
$K_{j}$ over $\phi^{-1}(U_j)\subset Y,$ satisfying
\begin{equation}\label{split}
K_{i} = K_{j} \otimes \phi^*(J_{ij}),
\end{equation}
where $J_{ij}= (s_i, s_j)^*J$ is the fixed local line bundle
over $U_{ij}.$

It follows that locally, $\Psi^\ZZ_J$ is a family of pseudodifferential operators acting fibrewise on $\phi^{-1}(U_i)$,
$$\Psi^{\ZZ}_J(\phi^{-1}(U_i); \hom(\bbE)) = \Psi^{\ZZ}(\phi^{-1}(U_i); \hom(\bbE\otimes K_i))$$
and by equation \eqref{split} that it fails to be a global family of pseudodifferential operators acting fibrewise on $Y$, 
since by \eqref{split}, there is no global line bundle on $Y$ that restricts to each of the local line bundles $K_j$.
Therefore the gerbe corresponding to the projective bundle $\Psi^{\ZZ}_J(Y/X; \hom(\bbE))$, which is the obstruction 
to finding a global line bundle on $Y$ restricting to each of the local line bundles $K_j$, is exactly
the collection of line bundles $J_{ij} = (s_i, s_j)^*J$ on the double overlaps $U_i\cap U_j$, proving the proposition.

\end{proof}

\subsection{Smooth Azumaya bundle for a sum of decomposable elements}\label{decomp} 
Here we outline an extension of the geometric setting in \cite{MMS2} which defines a smooth Azumaya bundle whose Dixmier-Douady invariant is a sum of decomposable classes. This is essential, since a general element in the cup product $\bfH^2(X;\bbZ) \cup \bfH^1(X;\bbZ) \subset \bfH^3(X;\bbZ) $
is of this form.

The data we use to define a smooth Azumaya bundle is:-
\begin{itemize}
\item 
A smooth function 
\begin{equation}
u\in\CI(X;\UU(1)^N)
\label{mms2002.115}\end{equation}
the homotopy class of which represents $\alpha\in \bfH^1(X,\bbZ^N).$\\
Equivalently, $u$ defines a regular covering space 
\begin{equation}
\xymatrix{
\bbZ^N\ar[r]& \hat X\ar[d]^\tau\\
& X
}
\label{mms2002.86}\end{equation}

\item A principal torus bundle bundle (later with connection)
\begin{equation}
\xymatrix{
\UU(1)^N\ar[r]& P\ar[d]^p\\
& X
}
\label{mms2002.87}\end{equation}
with Chern class $\beta \in \bfH^2(X;\bbZ^N).$
\item A smooth fiber bundle of compact manifolds
\begin{equation}
\xymatrix{
Z\ar[r]&Y\ar[d]^\phi\\
& X
}
\label{mms2002.59}\end{equation}
such that $\phi^*\beta =0$ in $\bfH^2(Y;\bbZ^N).$
\item An explicit global trivialization 
\begin{equation}
\gamma:\phi^*(P)\overset{\simeq}\longrightarrow Y\times\UU(1)^N.
\label{mms2002.88}\end{equation}
\end{itemize}
 The data \eqref{mms2002.115} -- \eqref{mms2002.88} are shown
below to determine an infinite rank {\em smooth Azumaya bundle}, which we
denote $\smA(\gamma).$ It has fibres isomorphic to the algebra of
smoothing operators on the fibre, $Z,$ of $Y$ with Schwartz kernels
consisting of the smooth sections of a line bundle $J(\gamma)$ over $Z^{[2]}.$
The completion of this algebra of `smoothing operators' to a bundle with
fibres modelled on the compact operators has Dixmier-Douady invariant
$\langle\alpha, \beta\rangle \in \bfH^3(Y;\bbZ),$ where 
$\langle , \rangle$ is the pairing $\bfH^2(X; \bbZ^N) \times \bfH^1(X; \bbZ^N)
\longrightarrow \bfH^3(X; \bbZ)$ given by cup product and the choice of $\ZZ$-bilinear pairing
$\bbZ^N \times \bbZ^N \mapsto \bbZ$ given by $(m, \ell) \mapsto m_1\ell_1 + \ldots m_N \ell_N$.

An explicit trivialization of the lift, $\gamma,$
as in \eqref{mms2002.87} is equivalent to a global section which is the
preimage under $\gamma$ of the identity element of the torus $\UU(1)^N$:
\begin{equation}
s':Y \longrightarrow \phi^*(P).
\label{mms2002.56}\end{equation}

Over each fiber of $Y,$ the image is fixed so this determines a map
$$
s(z_1,z_2)=s'(z_1)(s'(z_2))^{-1}
$$
which is well-defined on the fiber product and is a groupoid character:
\begin{equation}
\begin{gathered}
s:Y^{[2]}\longrightarrow
\UU(1)^N,\\
s(z_1,z_2)s(z_2,z_3)=s(z_1,z_3)\ \forall\ z_i\in Y\Mwith
\phi(z_i)=x,\ i=1,2,3,\ \forall\ x\in X.
\end{gathered}
\label{mms2002.57}\end{equation}
Conversely one can start with a unitary character $s$ of the groupoid $Y^{[2]}$ and
recover the principal torus bundle $P$ as the associated bundle
\begin{equation}
\begin{gathered}
P=Y\times\UU(1)^N/\simeq_s,\\
(z_1,t)\simeq_s(z_2,s(z_2,z_1)t)\ \forall\ t\in\UU(1)^N,\ \phi(z_1)= \phi(z_2).
\end{gathered}
\label{mms2002.58}\end{equation}

Now, let $Q=Y^{[2]}\times_X\hat X$ be the fiber product of $Y^{[2]}$ and
$\hat X,$ so as a bundle over $X$ it has typical fiber $Z^2\times\bbZ^N;$ it
is also a principal $\bbZ^N$-bundle over $Y^{[2]}.$ The data above determines
an action of $\bbZ^N$ on the trivial bundle $Q\times\bbC$ over $Q,$
namely  
\begin{equation}
T_n:(z_1,z_2,\hat x;w)\longrightarrow (z_1,z_2,\hat x+n,\langle\langle s(z_1,z_2), n \rangle\rangle w)\
\forall\ n\in \bbZ^N,
\label{mms2002.63}\end{equation}
where $\langle\langle \cdot , \cdot\rangle\rangle \colon \UU(1)^N \times \bbZ^N 
\longrightarrow \UU(1)$ denotes the Pontrjagin duality pairing between
the Pontrjagin dual groups $\UU(1)^N$ and $\bbZ^N$.

Let $J$ be the associated line bundle over $Y^{[2]}$ 
\begin{equation}
J =(Q\times\bbC)/\simeq,\quad (z_1,z_2,\hat x;w)\simeq T_n(z_1,z_2,\hat x;w)\
\forall\ n\in \bbZ^N
\end{equation}
The fiber of $J$ at $(z_1, z_2) \in Y^{[2]}$ such that $\phi(z_1)=\phi(z_2)=x$ is 
\begin{equation}
J_{z_1,z_2}=\hat X_x\times\bbC/\simeq,\quad (\hat x+n,w)
\simeq (\hat x,\langle\langle s(z_1,z_2), -n \rangle\rangle w).
\label{mms2002.74}\end{equation}
Note that this primitive line bundle \emph{does depend} on the trivialization data in
\eqref{mms2002.87}; we will therefore denote it $J(\gamma).$ 

This line bundle is \emph{primitive} in the sense that under
lifting by the three projection maps
\begin{equation}
\xymatrix{Y^{[3]}\ar@<2ex>[r]^{\pi_S}
\ar[r]^{\pi_C}
\ar@<-2ex>[r]^{\pi_F}& Y^{[2]}}
\label{mms2002.138}\end{equation}
(corresponding respectively to the left two, the outer two and the right
two factors) there is a natural isomorphism 
\begin{equation}
\pi_S^*J\otimes\pi_F^*J=\pi_C^*J.
\label{mms2002.139}\end{equation}
%
which gives the identification 
\begin{equation}
J_{(z,z'')}\otimes J_{(z'',z')}\simeq J_{(z,z')}.
\label{mms2002.191}\end{equation}
%

As remarked above, $J(\gamma),$ 
depends on the particular global trivialization \eqref{mms2002.87}. Two
trivializations, $\gamma _i,$ $i=1,2$ as in \eqref{mms2002.87} determine
\begin{equation}
\gamma _{12}:Y\longrightarrow \UU(1)^N,\ \gamma _{12}(y)\gamma_2(y)=\gamma_1(y)
\label{mms2002.199}\end{equation}
which fixes an element $[\gamma_{12}]\in \bfH^1(Y;\bbZ^N)$ and hence a line
bundle $K_{12}$ over $Y$ with Chern class $\langle [\gamma_{12}], 
[\phi^*\alpha]\rangle \in \bfH^2(Y;\bbZ).$ Then 
\begin{equation}
J(\gamma _2)\simeq (K_{12}^{-1}\boxtimes K_{12})\otimes J(\gamma _1)
\label{mms2002.201}\end{equation}
with the isomorphism consistent with primitivity.

That is, we have constructed a finite dimensional bundle gerbe $(Y/X;J)$
%
%
with Dixmier-Douady invariant 
$$ \delta(Y/X;J) = \langle \alpha,\beta\rangle \in \bfH^3(X;\bbZ).$$ 
%
As in the previous subsection, we can also define projective bundles of pseudodifferential operators
in this setup.

\subsection{A canonical example}

The following extends the discussion of an example  in \cite{MMS2}.
In particular,  let $\phi: Y \rightarrow X$ be a fibre bundle of compact  
manifolds,
with typical fiber a compact Riemann surface $\Sigma_g$
of genus $g\ge 2$. Then $T(Y/X)$ is an oriented rank 2 bundle over $Y.$
Define $\beta = \phi_*(e \cup e) \in \bfH^2(X, \bbZ)$, where
$e := e(T(Y/X)) \in \bfH^2(Y, \bbZ)$ is the Euler class of $T(Y/X)$. By  
naturality of this
construction, $\beta = f^*(e_1)$, where $e_1 \in \bfH^2(B{\rm Diff} 
(\Sigma_g), \bbZ)$
and $f\colon X \to B{\rm Diff}(\Sigma_g)$ is the classifying map for $ 
\phi:Y \to X$.
$e_1$ is known as the universal first Mumford-Morita-Miller class,
and $\beta$ is the first Mumford-Morita-Miller class of $\phi: Y  
\rightarrow X$, cf\@.~Chapter 4 in \cite{Morita}.
Therefore by Lemma 14 in \cite{MMS2}, we have the following.

\begin{lemma}
In the notation above, let $\phi: Y \rightarrow X$ be a fibre bundle of  
compact manifolds,
with typical fiber a compact Riemann surface $\Sigma_g$
of genus $g\ge 2$, and let $\beta \in \bfH^2(X, \bbZ)$ be a multiple of the
first Mumford-Morita-Miller class of $\phi: Y \rightarrow X$.
Then $\phi^*(\beta) = 0$ in $\bfH^2(Y, \bbZ)$.
\end{lemma}

Let $\phi:Y\longrightarrow X$
be as above, and $X$ be a closed Riemann surface.
Then Proposition 4.11 in \cite{Morita} asserts that $\langle e_1, [X] 
\rangle = {\rm Sign}(Y)$, where
$ {\rm Sign}(Y)$ is the signature of the 4-dimensional manifold $Y$,  
which is originally a result of
Atiyah, cf. \cite{Morita}. As a consequence, Morita is able to produce infinitely many  
surface bundles $Y$ over
$X$ that have non-trivial  first Mumford-Morita-Miller class.

On the other hand, given any $\beta \in \bfH^2(X, \bbZ)$, we know that  
there is
a fibre bundle of compact manifolds
$\phi: Y \rightarrow X$ such that $\phi^*(\beta)=0$ in $\bfH^2(Y, \bbZ)$.
In fact we can choose $Y$ to be the total space of a principal ${\rm U} 
(n)$ bundle
over $X$ with first Chern class $\beta$.  Here we can also replace $ 
\UU(n)$
by any compact Lie group $G$ such that $\bfH^1(G, \bbZ)$ is nontrivial  
and torsion-free,
such as the torus $\bbT^n$.

\begin{lemma}
Let $\phi: Y \rightarrow X$ be a fibre bundle of compact manifolds with  
typical fiber a compact Riemann surface
$\Sigma$ of genus $g\ge 2$ and $\beta \in \bfH^2(X, \bbZ)$.
Let $\pi: P \to X$ be a principal ${\rm U}(n)$-bundle whose first  
Chern class is $\beta$. Then the fibred
product  $\phi \times \pi : Y\times_X P \to X$ is a fiber bundle with  
typical fiber $\Sigma \times {\rm U}(n)$,
and has the property that $(\phi \times \pi )^*(\beta) = 0$ in $\bfH^2(Y 
\times_X P, \bbZ).$
\end{lemma}

This follows from the obvious commutativity of the
  following diagram,
  \begin{equation}
\begin{CD}
Y\times_X P  @>{pr_1}>> Y \\
       @V{pr_2}VV          @VV{\phi}V     \\
P   @>\pi>>  X.
\end{CD}\end{equation}

The construction of the universal fibre bundle of Riemann surfaces which we
will describe next, is well known, cf. \cite{ASZ, AS84, F86}.  Let $\Sigma$
be a compact Riemann surface of genus $g$ greater than $1,$ $\fM_{(-1)}$ the
space of all hyperbolic metrics on $\Sigma$ of curvature equal to $-1,$ and
${\rm Diff}_+(\Sigma)$ the group of all orientation preserving
diffeomorphisms of $\Sigma$. Then the quotient
$$
\fM_{(-1)}/{\rm Diff}_+(\Sigma) =  \cM_g
$$ 
is a noncompact orbifold, namely the moduli space of Riemann surfaces of 
genus equal to $g$.
The fact that $\cM_g$ has singularities can be dealt with in several ways, 
for instance by going to a finite smooth cover, and the noncompactness of $\cM_g$
can be dealt with for instance by considering compact submanifolds. We will 
however not deal with these delicate issues in the discussion below.
The group ${\rm Diff}_+(\Sigma)$
also acts on $\Sigma \times \fM_{(-1)}$ via $g(z, h) = (g(z), g^*h)$
and the resulting smooth fibre bundle, 
\begin{equation}
\pi: Y = (\Sigma \times \fM_{(-1)})/{\rm Diff}_+(\Sigma) \longrightarrow 
\fM_{(-1)}/{\rm Diff}_+(\Sigma) =\cM_g
\label{UniBun}
\end{equation}
is the {\em universal bundle} of genus $g$ Riemann surfaces. The classifying map for
\eqref{UniBun} is the identity map on $\cM_g$ so $\pi$ is maximally nontrivial in a sense made
precise below.  

As before, let
$$
e_1=e_1(Y/\cM_g) = \pi_*(e\cup e) \in \bfH^2(\cM_g; \bbZ)
$$
be the first Mumford-Morita-Miller class of $\pi\colon 
 Y \to \cM_g.$

A theorem of Harer \cite{Harer, Morita} asserts that:
\begin{align*}
\bfH^2(\cM_g; \bbQ) & = \bbQ(e_1);\\ 
\bfH^1(\cM_g; \bbQ) & =\{0\}.
\end{align*}
Our next goal is to define 
a line bundle $\cL$ over $\cM_g$ such that $c_1(\cL) =k e_1$ for some 
$k\in \ZZ$. This line 
bundle then automatically has the property that $\pi^*(\cL)$ is trivializable since $e_1$ 
is a characteristic class of the fibre bundle $\pi:Y\longrightarrow\cM_g$.
This is exactly the data needed to define a projective family of Dirac
operators. The line bundle $\cL$ turns out to be a power of the determinant
line bundle of the virtual vector bundle $\Lambda$ known as the Hodge
bundle, which is defined using the Gysin map in K-theory
$$
\Lambda=\pi_!(T(Y/ \cM_g)) \in K^0( \cM_g).
$$
Then $\det(\Lambda)$ is a
line bundle over $\cM_g.$  Next we need the following special
Grothendieck-Riemann-Roch (GRR) calculation, cf. \cite{MMS2} Appendix C,
Lemma 15.

\begin{lemma}
In the notation above, one has the following identity of first Chern classes,
$$
 c_1( \pi_!(T(Y/\cM_g)) 
 =\frac{13}{12} \pi_*(c_1(T(Y/\cM_g))^2).
 $$
 \end{lemma}
 
 
Observing that $c_1(T(Y/\cM_g) = e$ and 
$$
c_1( \pi_!(T(Y/\cM_g))  = c_1(\Lambda) = c_1(\det(\Lambda)),
$$
the lemma above shows that $c_1(\det(\Lambda)) = \frac{13}{12} e_1.$
Setting $\cL = \det(\Lambda)^{\otimes 12},$ we obtain, cf. \cite{MMS2} Appendix C, Corollary 2.

\begin{corollary}
In the notation above, $\cL$ is a line bundle over $\cM_g$
and one has the following identity:
$$
 c_1(\cL) 
 = 13 e_1.
 $$
 \end{corollary}

We next construct a canonical projective family of Dirac operators
on the Riemann surface $\Sigma$ with fixed choice of spin structure. This family is different to the one constructed in 
\cite{MMS2} Appendix C.
We enlarge the parametrizing space $\cM_g$ by taking the product with the Jacobian variety
${\rm Jac}(Y)$, which is the smooth variety of 
all unitary characters of the fundamental group $\pi_1(Y)$ of $Y$.

Construct a tautological line bundle $\cP$ over $Y \times {\rm Jac}(Y)$ as follows.
Consider the free action 
of $ \pi_1(Y)$,
\begin{align*}
 \pi_1(Y) \times \widetilde Y \times {\rm Jac}(Y) \times \bbC & \to \widetilde Y \times {\rm Jac}(Y) \times \bbC\\
(\gamma, (y, \chi, z)) & \to (y. \gamma, \chi, \chi(\gamma) z),
\end{align*}
where $\widetilde Y$ is the universal covering space of $Y$.
Then $\cP$ is defined to be the quotient space,
$$
\cP = \left( \widetilde Y \times {\rm Jac}(Y) \times \bbC\right)/ \pi_1(Y) 
$$ 

Consider now the fibre bundle $\pi \times {\rm Id}: Y \times  {\rm Jac}(Y) \rightarrow \cM_g \times {\rm Jac}(Y)$,
with typical fibre the Riemann surface $\Sigma$. The fibre bundle is endowed with the tautological 
line bundle $\cP \rightarrow Y \times {\rm Jac}(Y)$ over the total space of the fibre bundle. 
Applying the main construction in \cite{MMS2}, we get
a primitive line bundle $J\longrightarrow (Y \times  {\rm Jac}(Y))^{[2]}$. 
By the construction at the end of \cite{MMS2} \S5, we obtain
a projective family of Dirac operators $\eth_{\cP \otimes J}$ on the Riemann surface
$\Sigma$, parametrized by 
$\cM_g \times {\rm Jac}(Y)$, having analytic index,
$$
\Index_a(\eth_{\cP \otimes J}) \in K^0(\cM_g \times {\rm Jac}(Y);  e_1 \cup a),
$$
where $a \in \bfH^1({\rm Jac}(Y); \bbZ)$ and the right hand side denotes the twisted $K$-theory.

In the discussion above, we need to know that $\dim( {\rm Jac}(Y))>0$. We will establish this 
in the case when the genus $g\ge 1$. By the Leray-Serre spectral sequence for the 
fibre bundle $\Sigma \hookrightarrow Y \stackrel{\pi}{\rightarrow} \cM_g$, one has the 
5-term exact sequence of low degree homology groups, cf. Corollary 9.14 \cite{DK},
\begin{equation}\label{5-term}
\bfH_2(Y)  \stackrel{\pi_*}{\rightarrow} \bfH_2( \cM_g) \stackrel{\tau}{\rightarrow} \bfH_0(\cM_g, \bfH_1(\Sigma))
{\rightarrow} \bfH_1(Y) \stackrel{\pi_*}{\rightarrow} \bfH_1(\cM_g) {\rightarrow} 0
\end{equation}
where $\tau$ denotes the transgression map.

By  \cite{Mac}, one knows  that  $\cM_g$ is simply-connected,  therefore 
$$\bfH_0(\cM_g, \bfH_1(\Sigma)) \cong \bfH_1(\Sigma) \cong \bbZ^{2g},$$ and
by the Hurewicz theorem, 
$\bfH_1(\cM_g)=0$. By the universal coefficient theorem and by \cite{Morita},
$\bfH_2(\cM_g) \cong  \bfH^2(\cM_g, \bbZ) \cong \bbZ$. 
Therefore the 5-term exact sequence in equation \eqref{5-term} reduces to the 4-term exact sequence,
\begin{equation}\label{4-term}
\bfH_2(Y)  \stackrel{\pi_*}{\rightarrow} \bfH_2( \cM_g) \stackrel{\tau}{\rightarrow} \bfH_1(\Sigma)
{\rightarrow} \bfH_1(Y) \stackrel{\pi_*}{\rightarrow}   0
\end{equation}
Therefore, $\bfH_1(Y) \cong \bfH_1(\Sigma)/ {\rm Image}(\tau)$ has rank $\ge 2g-1>0$ (since 
the rank of ${\rm Image}(\tau)$ is $\le 1$,
therefore $\dim( {\rm Jac}(Y)) \ge 2g-1>0$, whenever $g\ge 1$.

We summarise the above as follows, also using the main result in \cite{MMS2}.

\begin{proposition}
Let $\Sigma$ be a Riemann surface of genus $g\ge 2$ and $Y\longrightarrow \cM_g$ be the
canonical family of hyperbolic metrics on $\Sigma$ of curvature equal to $-1$.
Consider the fibre bundle $\pi \times {\rm Id}: Y \times  {\rm Jac}(Y) \rightarrow \cM_g \times {\rm Jac}(Y)$,
with typical fibre the Riemann surface $\Sigma$. The fibre bundle is endowed with the tautological 
line bundle $\cP \rightarrow Y \times {\rm Jac}(Y)$ over the total space of the fibre bundle. 
Applying the main construction in \cite{MMS2}, we get
a primitive line bundle $J\longrightarrow (Y \times  {\rm Jac}(Y))^{[2]}$. 
By the construction at the end of \cite{MMS2} \S5, we obtain
a projective family of Dirac operators $\eth_{\cP \otimes J}$ on the Riemann surface
$\Sigma$, parametrized by 
$\cM_g \times {\rm Jac}(Y)$, having analytic index,
$$
\Index_a(\eth_{\cP \otimes J}) \in K^0(\cM_g \times {\rm Jac}(Y);  e_1 \cup a),
$$
where $a \in \bfH^1({\rm Jac}(Y); \bbZ)$ and the right hand side denotes the twisted $K$-theory.
\end{proposition} 

\appendix
\section{Invertible Fourier integral operators}

Fourier integral operators, as introduced by H\"ormander in \cite{HoFIO},
see also \cite{Ho4}, are operators with Schwartz' kernels which are
Lagrangian distributions. The distributions associated to a conic
Lagrangian submanifold, $\Lambda \subset T^*M\setminus O,$ are defined
through local (really microlocal) parameterizations, they are then given by
oscillatory integrals over the fibres of the parameterizations. The case of
primary interest here, corresponds to $\Lambda =\operatorname{graph}'(\chi)\subset
(T^*Z\setminus O)^2$ being the twisted graph of a canonical,
i.e.\ homogeneous symplectic, diffeomorphism. In case $\chi$ is close to
the identity, in the \ci\ topology, it is possible to give a \emph{global}
parameterization (as for a pseudodifferential operator) which presents the
kernel in terms of a single oscillatory integral. The group of canonical
transformations, $\Can(Z)$ is naturally identified with the group of
contact transformations of $S^*Z$ and is a Fr\'echet Lie group, with the
same \ci\ topology as the full group of diffeomorphisms of $S^*Z.$ More
precisely $\Can(Z)$ is a Fr\'echet manifold modelled on $\CI(S^*Z),$
realized as the global space of smooth sections of the trivial bundle over
$S^*Z$ corresponding to functions homogeneous of degree $1$ on $T^*Z.$

\begin{proposition}\label{20.10.2012.1} The choice of a Riemann metric on
  $Z$ gives an identification of a neighbourhood of the identity in
  $\Can(Z)$ with a neighbourhood of $0$ in $\CI(S^*Z)$ under which the
  Lie algebra is mapped to $\CI(S^*Z)$ with the normalized Poisson bracket.
\end{proposition}

\begin{proof} The exponential map corresponding to a choice of Riemann
  metric on $Z$ gives a normal fibration, a collar neighbourhood, of the
  diagonal in $Z^2$ in the form 
\begin{equation}
TZ\supset U\ni(z,v)\longrightarrow (\exp_z(v),z)\in U'\subset Z^2
\label{20.10.2012.2}\end{equation}
which is a diffeomorphism from $U,$ an open neighbourhood of the zero
section, to $U',$ an open neighbourhood of the diagonal in $Z^2.$ Under
this map the cotangent bundle is identified with the subset of the fibre
product $TZ\oplus T^*Z\oplus T^*Z$ projecting to $U$ in the first factor 
\begin{equation}
T^*Z\subset TZ\oplus T^*Z\oplus T^*Z,\ (\exp_z(v),z,\tau_z\xi,\eta)\longmapsto (z,v,\xi,\eta)
\label{20.10.2012.3}\end{equation}
where $\tau_z\xi\in T^*_{\exp_z(v)}Z$ is obtained by parallel transport
from $z$ along the geodesic defining $\exp_z(v).$

For $F\in\Con(Z)$ in a neighbourhood of the identity with respect to some
$\cC^2$ norm, the image of each fibre $T^*_zZ\setminus 0$ is necessarily a
conic Lagrangian submanifold $\Lambda _z\subset T^*Z\setminus O$ which is
close to $T^*Z.$ Consider the reversed graph $\Gamma(F)=\{(F(\sigma) ,\sigma
);\sigma \in T^*Z\}\subset T^*Z^2.$ Then in terms of \eqref{20.10.2012.3}
$\Gamma _z(F)=\Gamma(F)\cap T^*Z\times T^*_zZ$ projects diffeomorphically
onto the first factor of $T^*Z\setminus0$ and so may be written as the
range of a homogeneous smooth map $\Phi_z:T^*_zZ\setminus 0\longrightarrow
T_zZ,$ i.e. 
\begin{equation}
\pi(\Gamma _z(F))=\{(\Phi(\eta),\eta)\}.
\label{20.10.2012.4}\end{equation}
As a homogeneous Lagrangian submanifold both the symplectic form, which is
$dv\wedge d\xi$ and the canonical form $\xi\cdot dv$ vanish on
$\pi(\Gamma_z(F))$ for each $z.$ Thus if the $\xi$ are used as coordinates on
$\pi(\Gamma _z(F))$ so $v=v(\xi)$ then
\begin{equation}
\phi(z,\xi)=v(\xi)\cdot\xi\Longrightarrow d_\xi\phi(x,\xi)=v(\xi).
\label{20.10.2012.10}\end{equation}
Thus in fact $\Phi$ is given as the
gradient of a uniquely defined smooth function $\phi(z,\xi)$ which is
homogeneous of degree $1$ in $\xi$ and so projects to a smooth function
$\tilde\phi=\phi(z,\xi)/|\xi|_g\in\CI(S^*Z).$ This identifies a
neighbourhood of the identity in $\Can(Z)$ with a neighbourhood of $0$ in
$\CI(S^*Z)$  
\begin{equation}
\pi(\Gamma_z(F))=\{d_\xi(\phi(z,\xi);\xi\in T^*_zZ\}.
\label{20.10.2012.5}\end{equation}
Note that the collection of these conic Lagrangian submanifolds, of
$T^*Z\times Z$ as a bundle over the second factor of $Z,$ does determine
$F.$ Indeed, the reverse twisted graph of $F$ can be recovered from the
phase function defined over $U'$ as 
\begin{equation}
\psi(z,v,\eta)=v\cdot\eta-\phi(z,\eta)
\label{20.10.2012.6}\end{equation}
where the first term is the pairing between $T_zZ$ and $T^*zZ.$

Although this construction does depend on the normal fibration around the
diagonal, \eqref{20.10.2012.2}, changing this induces a bundle
transformation corresponding to the fact that the structure groupoid of
$\Can(Z)$ can be reduced to the fibre-preserving local diffeomorphisms of $TZ.$

Since $[0,1]\ni t\longmapsto t\phi\in\CI(S^*Z)$ connects $\phi$ to $0$ it
follows that any canonical transformation near the identity is given as the
parameter-dependent integral of the corresponding Hamilton vector fields
$tH_{\tilde\phi}$ and hence that the Lie algebra of $\Can(Z)$ may be
identified with the Poisson bracket projected to $\CI(S^*Z),$  
\begin{equation}
[\tilde\phi_1,\tilde\phi_2]=\{\phi_1,\phi_2\}/|\xi|.
\label{20.10.2012.7}\end{equation}
\end{proof}

This parameterization of canonical transformations near the identity gives
a section of the group of invertible Fourier integral operators -- essentially
as in \cite{Adams}.

\begin{proposition}\label{20.10.2012.8} Given the choice of a Riemann metric on
  a compact manifold $Z,$ of a cutoff $\chi\in\CIc(U'),$
  equal to $1$ near the diagonal and of a connection on $V$ to define parallel
  transport of the identity operator on $V,$ the oscillatory integral 
\begin{equation}
K(\phi)=(2\pi)^{-n}\int e^{i(v\cdot\xi-\phi(z,\xi))}\chi(v,z)\Id(v,z)d\xi dg_z,
\label{20.10.2012.9}\end{equation}
defines a section of the bundle of invertible Fourier integral operators
over the neighbourhood of the identity in $\Can(Z)$ corresponding to
$\phi|\xi|^{-1}\in\CI(S^*Z)$ near $0$ with respect to a sufficiently high
$\cC^k$ norm.
\end{proposition}

\begin{proof} The discussion above shows that the phase function 
\begin{equation}
\psi(v,z,\xi)=v\cdot\xi-\phi(z,\xi)\Mon TZ\oplus T^*Z
\label{20.10.2012.11}\end{equation}
as a bundle over $U$ parameterizes the twisted graph of the canonical
transformation in the sense introduced by H\"ormander
\cite{Ho4}. Thus \eqref{20.10.2012.9} is indeed a Fourier integral
operator associated, i.e.\ projecting to, the canonical transformation
used to define $\phi.$ For $\phi$ close to zero the graph of the
canonical transformation is close to the identity and \eqref{20.10.2012.9}
is therefore elliptic. The adjoints of the $K(\phi),$ with respect to some
fixed choice of inner product on $V$ and density on $Z,$ are necessarily
Fourier integral operators associated to the inverse transformations and
the product 
\begin{equation}
D(\phi)=K(\phi)^*K(\phi)
\label{PostArx.1}\end{equation}
is an elliptic pseudodifferential operator with kernel given explicitly by the
an oscillator integral. Application of the stationary phase lemma
shows that $D$ is a smooth map from a neighbourhood of zero in $\CI(S^*Z)$
into a neighbourhood of the identity in the elliptic
pseudodifferential operators. For a possibly smaller
neighbourhood it therefore lies in $G^0(Z;V).$ Thus $K(\phi)$ gives a section of
$\GL(\cF^0(Z;V))$ near the identity in $\Can(Z),$ as desired.
\end{proof}

If $E_0$ is the range of $K$ on a neighbourhood of $0$ in $\CI(S^*Z),$ and
hence a neigbhourhood $D_0$ of $\Id\in\Can_0(Z),$ then composing on the
left with invertible pseudodifferential operators, $G^0(Z;V)\cdot E_0$
certainly has the structure of a principal $G^0(Z;V)$-bundle over $D_0.$ To
extend this to the whole of $G_{\dag}(\cF^0(Z;V)),$ the part of the group
corresponding to the component of the identity $\Can_0(Z)\subset\Can(Z),$
it suffices to see that there is an element of $G_{\dag}(\CF^0(Z;V))$
corresponding to a given element $\chi\in\Can_0(Z).$ By assumption there is
a smooth curve $\chi_t$ in $\Can(Z)$ starting at the identity and with
end-point $\chi.$ It follows from smoothness that the Maslov bundle is
trivial, so there is no obstruction to the existence of an elliptic symbol
for such an operator. Indeed, again using the smooth homotopy a smooth
curve of elliptic operators, $F_t\in\cF^0(Z;V),$ starting at the identity
can be constructed covering the curve $\chi_t.$ This curve must have index
zero and hence the end-point, $F_1,$ with canonical transformation $\chi$
can be perturbed by the addition of a smoothing operator to be invertible,
i.e.\ to lie in $G_{\dag}(\cF^0(Z;V)).$ It follows that
$G_{\dag}(\cF^0(Z;V))$ is a principal bundle over $\Can_0(Z)$ with
structure group $G^0(Z;V).$

\section{Homotopy groups of (projective) Invertible Fourier Integral Operators}

In this appendix, we calculate the homotopy groups of the group of invertible pseudodifferential operators 
of order zero acting on a stable vector bundle over a Riemann surface $\Sigma_g$
of genus $g \ge 2$, as well as the homotopy groups of the group of invertible Fourier integral operators 
of order zero and its projectivization. As a consequence, we show that whereas there are no non-trivial fibre bundles over 
spheres $S^n,\, n\ge 2$ with typical fibre a Riemann surface $\Sigma_g$ of genus $g\ge 2$,
there are infinitely many topologically distinct  principal $ \PGL(\cF^0(\Sigma_g))$ bundles over $S^n,\, n\ge 2$, all of which 
are therefore purely infinite dimensional.

Let $\GL(\Psi^0(\Sigma_g; V))$ denote the group of invertible pseudodifferential operators 
of order zero acting on a vector bundle $V$ of rank $\ge 2$ over a Riemann surface $\Sigma_g$
of genus $g \ge 2$. 
Define the (stabilized) group $\GL(\Psi^0(\Sigma_g))$ of invertible pseudodifferential operators of order zero
as,
$$
\GL(\Psi^0(\Sigma_g)) = \lim_{n\to\infty} \GL(\Psi^0(\Sigma_g, \underline{\CC^n}))
$$
where $\underline{\CC^n}$ denotes the trivial bundle of rank $n$ on $\Sigma_g$.

Then by a result of Rochon \cite{Rochon}, the homotopy groups of $\GL(\Psi^0(\Sigma_g))$ are given by
$$
\pi_k(\GL(\Psi^0(\Sigma_g))) \cong \left\{\begin{array}{l} K^1(S^*\Sigma_g)/\ZZ \quad \text{if} \,\, k \,\, \text{is even}; \\[+7pt]
K^0(S^*\Sigma_g) \quad \text{if}\,\, k \,\,\text{is odd}.  \end{array}\right.   
$$
Since $S^*\Sigma_g \to \Sigma_g$ is a principal circle bundle with Euler characteristic $\chi(\Sigma_g) = 2-2g$,
we can use the Gysin sequence in $K$-theory to compute $K^\bullet(S^*\Sigma_g)$,
$$
\xymatrix{ K^0(\Sigma_g)  \ar[rr]^{\varepsilon^0} && K^0(\Sigma_g)  \ar[rr]^{\pi^*} && K^0(S^*\Sigma_g) \ar[d]^{\pi_!} & \\
   K^1(S^*\Sigma_g) \ar[u]^{\pi_!} && K^1(\Sigma_g) \ar[ll]^{\pi^*} && K^1(\Sigma_g) \ar[ll]^{\varepsilon^1} }\nonumber
$$
By a result of Emerson and Meyer \cite{EmersonMeyer}, $\varepsilon^0(x) = (2-2g) \rank(x) [{\bf 1}], \, x\in K^0(\Sigma_g)$ and 
$\varepsilon^0(y) =0, \, y\in K^1(\Sigma_g)$. Since $K^0(\Sigma_g) \cong \ZZ \oplus \ZZ$ and 
$K^1(\Sigma_g) \cong \ZZ^{2g}$, the Gysin sequence above breaks up into two short exact sequences,
\begin{align*}
 0\to \ZZ^{2g} \to & K^1(S^*\Sigma_g) \to \ZZ\to 0\\
 0\to  \ZZ \oplus \ZZ_{2g-2}\to & K^0(S^*\Sigma_g) \to \ZZ^{2g}\to 0
\end{align*}
Since $\ZZ$ and $\ZZ^{2g}$ are free abelian groups, these sequences split and we have,
$$
K^j(S^*\Sigma_g) \cong \left\{\begin{array}{l} \ZZ^{2g+1} \oplus \ZZ_{2g-2}  \quad \text{if} \,\, j=0,\\[+7pt]
 \ZZ^{2g+1}  \quad \text{if}\,\, j=1.  \end{array}   \right.
$$
We conclude that the homotopy groups of $\GL(\Psi^0(\Sigma_g))$ are,
$$
\pi_k(\GL(\Psi^0(\Sigma_g))) \cong \left\{\begin{array}{l} \ZZ^{2g} \quad \text{if} \,\, k \,\, \text{is even}; \\[+7pt]
 \ZZ^{2g+1} \oplus \ZZ_{2g-2}  \quad \text{if}\,\, k \,\,\text{is odd}. \end{array}  \right.
$$
\bigskip

Let $\GL(\cF^0(\Sigma_g; \underline{\CC^n}))$ denote the group of invertible Fourier integral operators 
of order zero acting on a trivial vector bundle of rank $n$ over a Riemann surface $\Sigma_g$
of genus $g \ge 2$, and $\Con_0(S^*\Sigma_g)$ denote the connected component of the group of 
contact diffeomorphisms of $S^*\Sigma_g$. Then by equation \eqref{27.August.2012.47}\, one has a fibration
$$
\GL(\Psi^0(\Sigma_g; \underline{\CC^n}))\longrightarrow  \GL(\cF^0(\Sigma_g; \underline{\CC^n})) \to \Con_0(S^*\Sigma_g).
$$
By taking the inductive limit $n\to \infty$, there is a fibration,
$$
\GL(\Psi^0(\Sigma_g))\longrightarrow  \GL(\cF^0(\Sigma_g)) \to \Con_0(S^*\Sigma_g),
$$
where 
$$
\GL(\cF^0(\Sigma_g)) =  \lim_{n\to\infty} \GL(\cF^0(\Sigma_g; \underline{\CC^n}))
$$
By a result of Banyaga \cite{Ban78} and Remarks \ref{Contact}, $\Con_0(S^*\Sigma_g)$ is homotopy equivalent to the circle $S^1$, so we conclude
that 
$$
\pi_k(\Con_0(S^*\Sigma_g)) \cong \left\{\begin{array}{l} \ZZ \quad \text{if} \,\, k=1; \\[+7pt]
0 \quad \text{if}\,\, k\ne 1.  \end{array} \right.  
$$
By the long exact sequence in homotopy and the calculations above, we see that 
$$
\pi_k( \GL(\cF^0(\Sigma_g))) = \pi_k(\GL(\Psi^0(\Sigma_g))) \quad \text{if}  \,\, k>1,
$$
and 
$$
0\to\pi_1( \GL(\Psi^0(\Sigma_g))) \longrightarrow  \pi_1(\GL(\cF^0(\Sigma_g))) \to \pi_1(\Con_0(S^*\Sigma_g))\to 0.
$$
That is, 
$$
0\to\ZZ^{2g+1} \oplus \ZZ_{2g-2}  \longrightarrow  \pi_1(\GL(\cF^0(\Sigma_g))) \to \ZZ\to 0.
$$
Now $\pi_1(\GL(\cF^0(\Sigma_g)))$ is an abelian group since the fundamental group of any topological group is abelian.
Also $\ZZ$ is a free abelian group, so the sequence above splits and we have
$$
 \pi_1(\GL(\cF^0(\Sigma_g))) \cong \ZZ^{2g+2} \oplus \ZZ_{2g-2}.
$$
To summarise, one has 
$$
 \pi_k(\GL(\cF^0(\Sigma_g)))   \cong \left\{\begin{array}{l} \ZZ^{2g} \quad \text{if} \,\, k>0 \,\, \text{is even}; \\[+7pt]
 \ZZ^{2g+1} \oplus \ZZ_{2g-2}  \quad \text{if}\,\, k>1 \,\,\text{is odd} \\[+7pt]
 \ZZ^{2g+2} \oplus \ZZ_{2g-2} \quad \text{if}\,\, k=1.
 \end{array}   \right.
$$
\bigskip

Consider the principal $\CC^*$-bundle 
\begin{equation}\label{eqn:circle}
\CC^* \to \GL(\cF^0(\Sigma_g)) \to \PGL(\cF^0(\Sigma_g)).
\end{equation}
Since $\pi_j(\CC^*)=0$ for $j\ne 1$ and $\pi_1(\CC^*)\cong \ZZ$, from the long exact sequence in homotopy, we see that 
$$
\pi_k(\PGL(\cF^0(\Sigma_g))) = \pi_k(\GL(\cF^0(\Sigma_g))) \quad \text{if} \,\, k\ge 3,
$$
and also 
$$
0\to \pi_2( \GL(\cF^0(\Sigma_g)) \to \pi_2(\PGL(\cF^0(\Sigma_g))) \stackrel{\partial}{\to} \ZZ  \to \pi_1( \GL(\cF^0(\Sigma_g))) \to 
\pi_1(\PGL(\cF^0(\Sigma_g))) \to 0.
$$
We conclude that 
$$
\ZZ^{2g}\cong \pi_2( \GL(\cF^0(\Sigma_g))) \hookrightarrow \pi_2(\PGL(\cF^0(\Sigma_g))) 
$$
and 
$$
 \ZZ^{2g+2} \oplus \ZZ_{2g-2} \cong \pi_1( \GL(\cF^0(\Sigma_g))) \twoheadrightarrow \pi_1(\PGL(\cF^0(\Sigma_g))).
$$
In fact, observe that since the kernel of the above surjective map has at most rank equal to 1, that 
$\pi_1(\PGL(\cF^0(\Sigma_g)))$ is infinite.
The boundary map $\partial$ is determined by the Euler class of the principal $\CC^*$-bundle  \eqref{eqn:circle}.

These results are used to prove the following,

\begin{proposition} Whereas there are no non-trivial fibre bundles over $S^n,\, n\ge 2$ with typical fibre $\Sigma_g, \, g\ge 2$,
there are infinitely many topologically distinct principal $ \GL(\cF^0(\Sigma_g))$ bundles over $S^n,\, n\ge 2$. Similarly 
there are infinitely many topologically distinct  principal $ \PGL(\cF^0(\Sigma_g))$ bundles over $S^n,\, n\ge 2$.
\end{proposition}

These principal $ \GL(\cF^0(\Sigma_g))$ and $ \PGL(\cF^0(\Sigma_g))$ bundles over $S^n,\, n\ge 2$ 
are all purely infinite dimensional, that is, they do not arise from any (finite dimensional)
fibre bundle over $S^n,\, n\ge 2$ with typical fibre $\Sigma_g, \, g\ge 2$.

\section{Factors of automorphy \& pseudodifferential algebra bundles} \label{secappCb}

In this section we study principal $\PGL(\cF^0(Z;V))$-bundles $\bfF$ and
the associated pseudodifferential algebra bundles $\bfPsi^\ZZ$ over a
compact manifold $X$ having contractile universal cover $\wX$ and relate
their construction to factors of automorphy. For example, any $X$ that is
a locally symmetric space $\Gamma\backslash G/K$ where $G$ is a reductive
Lie group with maximal compact subgroup $K$ and $\Gamma$ a uniform lattice
in $G$, has the desired property that the universal cover $\wX= G/K$ is
contractible.  We also express the Dixmier-Douady invariant of the
pseudodifferential algebra bundle $\bfPsi^\ZZ$ over $X$ in terms of the
factors of automorphy.

Let $\wF$ denote the lift of $\bfF$ to $\wX$.  Since $\wX$ is contractible,
it follows that $\wF$ is trivializable, i.e. $\wF\cong \wX \times
\PGL(\cF^0(Z;V))$.  Having fixed this isomorphism, we can define a
continuous map $j : \Gamma \times \wX \to \PGL(\cF^0(Z;V)) = {\rm
  Aut}(\Psi^\ZZ(Z;V))$ by the following commutative diagram,
\begin{equation}
\xymatrix{
                 \Psi^\ZZ(Z;V) = (\widetilde\bfPsi^\ZZ)_x  \ar[ddr]_{p} \ar[rr]^{j(\gamma, x)} && 
                 \Psi^\ZZ(Z;V) = (\widetilde\bfPsi^\ZZ)_{\gamma\cdot x}   \ar[ddl]^p                   \\  && \\ 
             &        (\bfPsi^\ZZ)_{p(x)} &      }
\end{equation}
Then 
\begin{equation}\label{eqn:auto}
j(\gamma_1\gamma_2, x) ^{-1}j(\gamma_1, \gamma_2 x) j(\gamma_2, x) = 1
\end{equation}
so $j$ is a factor of automorphy for the bundle $\bfPsi^\ZZ$. Conversely,
given a continuous map $j : \Gamma \times \wX \to \PGL(\cF^0(Z;V)) = {\rm
  Aut}(\Psi^\ZZ(Z;V))$ satisfying \eqref{eqn:auto}, we can define a
pseudodifferential algebra bundle,
\begin{equation}
\bfPsi^\ZZ_{j} = (\wX \times \Psi^\ZZ(Z;V))/\Gamma \to M
\end{equation}
where $\gamma \cdot (x, \xi) = (\gamma\cdot x, j(\gamma, x) \xi)$ for $\gamma \in \Gamma$ and 
$(x, \xi) \in \wX \times \Psi^\ZZ(Z;V)$. 

Given any two pseudodifferential algebra bundles, $\bfPsi^\ZZ_j, \bfPsi^\ZZ_{j'}$ over $X$
with factors of automorphy $j, j'$ respectively, and an isomorphism $\phi : \bfPsi^\ZZ_j \longrightarrow
\bfPsi^\ZZ_{j'}$, we get an induced isomorphism 
\begin{equation}
\widetilde \phi : \wX \times \Psi^\ZZ(Z;V) =  \widetilde\bfPsi^\ZZ_j 
\longrightarrow \widetilde\bfPsi^\ZZ_{j'} =  \wX \times \Psi^\ZZ(Z;V)
\end{equation}
given by $\widetilde\phi(x, \xi) = (x, u(x)\xi)$, 
where $u: \wX \to \PGL(\cF^0(Z;V))$ is continuous. 

Since $\widetilde\phi$ commutes with the action of $\Gamma,$ 
\begin{equation*}
\gamma\cdot  \widetilde\phi(x, \xi) = \widetilde\phi(\gamma(x, \xi))\text{
  and hence }
(\gamma\cdot x, j'(\gamma, x) u(x)\xi) = (\gamma\cdot x, u(\gamma\cdot x)j(\gamma, x)\xi).
\label{PostArx.3}\end{equation*}
Therefore 
\begin{equation}\label{eqn:autoequiv} 
 j'(\gamma, x) = u(\gamma\cdot x)j(\gamma, x) u(x)^{-1}
 \end{equation}
for all $x\in \wX$ and 
$\gamma \in \Gamma$. Conversely, two factors of automorphy $j, j'$ give rise to 
isomorphic algebra bundles $\bfPsi^\ZZ_j, \bfPsi^\ZZ_{j'}$ of compact operators if they
are related by \eqref{eqn:autoequiv} for some continuous function $u: \wX \to \PGL(\cF^0(Z;V))$.

Thus continuous sections of $\bfPsi^\ZZ$ can be viewed as continuous maps
$f\in C(\wX, \Psi^\ZZ(Z;V))$ satisfying the property, 
\begin{equation}\label{C14}
f(\gamma\cdot  x) = j(\gamma, x) f(x),\qquad \forall \gamma\in \Gamma, \, x\in \wX.
\end{equation}
For example, $f(x):=\sum_{\gamma \in \Gamma} j(\gamma, x)^{-1} F(\gamma \cdot x)$
converges uniformly on compact subsets of $\wX$ whenever $F: \wX \to \Psi^\ZZ(Z;V)$ is a compactly supported 
continuous function, and satisfies \eqref{C14}, therefore defining a continuous section of $\bfPsi^\ZZ$.

We next express the Dixmier-Douady invariant of the pseudodifferential algebra bundle
$\bfPsi^\ZZ$ over $X$ in terms of the factors of automorphy.
There is no obstruction to lifting the factor of automorphy $j : \Gamma \times \wX \to \PGL(\cF^0(Z;V)) = 
{\rm Aut}(\Psi^\ZZ(Z;V))$
to $\wh j: \Gamma \times \wX \to \GL(\cF^0(Z;V))$, because of our assumptions on $\wX$. However 
the cocycle condition \eqref{eqn:auto} has to be modified, 
\begin{equation}
\wh j(\gamma_1\gamma_2, x) ^{-1} \wh j(\gamma_1, \gamma_2 x) \wh j(\gamma_2, x) = 
\tau(\gamma_1, \gamma_2, x)\,,
\end{equation}
where $\tau : \Gamma \times \Gamma \times \wX \to \U(1)$.
There is no obstruction to lifting $\tau: \Gamma \times\Gamma \times\wX \to \U(1)$ 
to $\wh\tau :  \Gamma \times\Gamma \times\wX \to \RR$, however the cocycle condition
satified by $\tau$ has to be modified to $\delta\wh\tau(\gamma_1, \gamma_2, \gamma_3) = 
\eta(\gamma_1, \gamma_2, \gamma_3)$, where $\eta:  \Gamma \times\Gamma \times \Gamma
\to \ZZ$ is a $\ZZ$-valued 3-cocycle on $\Gamma$. One can show  $DD({\bfF}) = \delta([\tau']) = [\eta]$.
Thus, given a principal $\PGL(\cF^0(Z;V))$ bundle ${\bfF}$ on $X$, we have derived a cohomology 
class $ [\eta] \in \bfH^3(\Gamma, \ZZ) \cong \bfH^3(X, \ZZ)$ which is by standard arguments 
independent of the choices made. The relation with the previous discussion is that $[\eta] = [c]$. 

Notice that $\tau$ can be viewed as a 
continuous map $\tau': \Gamma\times\Gamma \to C(\wX, \U(1))$,
which is easily verified to be a $C(\wX, \U(1))$-valued 2-cocycle on $\Gamma$. 
Recall from standard group cohomology theory that equivalence 
classes of extensions of a group $\Gamma$ by an abelian 
group $C(\wX, \U(1))$ on which $\Gamma$ acts is in bijective correspondence
correspondence with the group cohomology with coefficients, 
$\bfH^2(\Gamma, C(\wX, \U(1)))$. 
We will first show that possible extensions $\wGamma$ of $\Gamma$ by $C(\wX, \U(1))$
are in bijective correspondence with elements of $\bfH^3(X, \ZZ)$ called the Dixmier-Douady 
invariant, and we will also compute $DD({\bfF}) \in \bfH^3(X, \ZZ)$ in our case. 
Now there is an exact sequence of abelian groups,
\begin{equation}
0 \to \ZZ \to C(\wX, \RR) \to C(\wX, \U(1)) \to 0 \,.
\end{equation}

This leads to a long exact sequence corresponding to change of coefficients,
\begin{equation}
\cdots \to \bfH^2(\Gamma,  C(\wX, \RR)) \to \bfH^2(\Gamma, C(\wX, \U(1))) \stackrel{\delta}{\to} 
\bfH^3(\Gamma, \ZZ) \to 
\bfH^3(\Gamma,  C(\wX, \RR))  \to \cdot
\end{equation}
Since $\Gamma$ acts freely on $\wX$ and $C(\wX, \RR)$ is an induced module,
it follows that $ H^j(\Gamma, C(\wX, \RR)) = 0$ for all $j>0$.  Therefore
$H^j(\Gamma, C(\wX, \U(1))) \cong H^{j+1}(\Gamma, \ZZ) = H^{j+1}(X, \ZZ)$
for all $j>0$, and in particular for $j=2$ as claimed.  In particular,
since $[\tau'] \in \bfH^2(\Gamma, C(\wX, \U(1)))$, we see that $DD({\bfF}) =
\delta([\tau'])=[\eta] \in \bfH^3(\Gamma, \ZZ) = \bfH^3(X, \ZZ)$ is the
Dixmier-Douady invariant of ${\bfF}$.

\end{document}